\documentclass{article}

\usepackage[utf8]{inputenc}
\usepackage{tikz}
\usepackage[T1]{fontenc}
\usepackage{dsfont}
\usepackage[a4paper, margin=2.7cm]{geometry}
\usepackage{amsmath, amsthm, amsfonts, amssymb}
\usepackage{mathrsfs}        
\usepackage{enumerate}
\usepackage{esint}
\usepackage{caption}
\usepackage{dsfont}
\usepackage{xcolor}
\usepackage[colorlinks=true,linkcolor=blue,citecolor=blue,urlcolor=blue,breaklinks]{hyperref}
\usepackage{bbm} 
\usepackage{thmtools}
\usepackage{natbib}
\usepackage{mathtools}
\usepackage{url}

\newcommand{\abs}[1]{\left\vert#1\right\vert}           
\newcommand{\ap}[1]{\left\langle#1\right\rangle}        
\newcommand{\fl}[1]{\left\lfloor#1\right\rfloor}        
\newcommand{\cl}[1]{\left\lceil#1\right\rceil}          
\newcommand{\norm}[1]{\left\Vert#1\right\Vert}          
    
\renewcommand{\div}{\operatorname{div}}       

\DeclareMathOperator{\pv}{p.v.}

\def\d{\,\mathrm{d}}

\def \ddt{\frac{\mathrm{d}}{\mathrm{d}t}}

\def\p{\partial}

\def\N{\mathbb{N}}
\def\R{\mathbb{R}}

\def\E{\mathbb{E}}

\def\M{\mathcal{M}}

\def\Y{\mathcal{Y}}
\def\X{\mathcal{X}}
\def\Var{{\textrm{Var}}\,}

\def\Id{{\textrm{Id}}\,}

\def\ird{\int_{\R^d}}
\def\d{\,\mathrm{d}}

\def\e{\varepsilon}
\def\f{\varphi}

\def\d{\,\mathrm{d}}

\def \ddt{\frac{\mathrm{d}}{\mathrm{d}t}}
\def \ddt{\frac{\mathrm{d}}{\mathrm{d}t}}

\def\p{\partial}

\def\:{\colon}                        

\def\Jn{J_\varepsilon^{t_1,\dots,t_n}}

\def\JJi{\hat{f}\Big(\frac{\xi\sigma_j}{n^{1/2s}\bar\sigma}\Big)}
\def\GGi{\hat{G}^s\Big(\frac{\xi\sigma_j}{n^{1/2s}\bar\sigma}\Big)}
\def\JJj{\hat{f}\Big(\frac{\xi\sigma_k}{n^{1/2s}\bar\sigma}\Big)}
\def\GGj{\hat{G}^s\Big(\frac{\xi\sigma_k}{n^{1/2s}\bar\sigma}\Big)}

\def\Jnorm{\mathfrak{f}_n}
\def\Jnormh{\hat{\mathfrak{f}_n}}

\newtheorem{teorema}{Theorem}[section]
\newtheorem*{teorema*}{Theorem}
\newtheorem{lemma}[teorema]{Lemma}
\newtheorem{proposizione}[teorema]{Proposition}

\theoremstyle{definition}
\newtheorem{definizione}[teorema]{Definition}
\theoremstyle{remark}
\newtheorem{osservazione}[teorema]{Remark}

\theoremstyle{plain}
\newtheorem{hyp}{Hypothesis}

\theoremstyle{definition}

\theoremstyle{remark}

\theoremstyle{remark}

\def\e{\varepsilon}

\def\aux{{L^1_k}}
\def\esp{\theta\delta}

\title{Time Asymptotics and Scaling Limits for a Nonlocal Fokker-Planck Equation with Heavy-Tailed Kernel}
\author{Niccolò Tassi\footnote{Departamento de Matem\'atica Aplicada, Universidad de Granada, Avenida de Fuentenueva S/N, 18071 Granada, Spain. \textit{E-mail address}: \texttt{\href{mailto:tassi@ugr.es}{tassi@ugr.es}}}}
\date{}

\AtEndDocument{\vspace{2cm}{\footnotesize
		\noindent
		\textsc{Niccolò Tassi. Departamento de Matemática Aplicada \& IMAG,
			Universidad de Granada. Avda. Fuentenueva S/N, 18071 Granada, Spain.}
		\textit{E-mail address}: \texttt{\href{mailto:tassi@ugr.es}{tassi@ugr.es}}
}}

\begin{document}

\maketitle

\begin{abstract}
We investigate the asymptotic behaviour of solutions of a class of nonlocal Fokker--Planck equations defined by  nonsingular, heavy-tailed convolution kernels and characterised by a scaling parameter $\e\in(0,1]$ and a fractional index $s\in(1/2,1)$. By employing a suitable version of the generalised central limit for heavy-tailed distributions and the use of Harris's theorem, we prove exponential convergence to the equilibrium with a rate that is independent of both $\e$ and $s$. 
This allows us to show uniform--in--time convergence for both $\e\to 0$ and $s\to1$ recovering the limiting equations.
\end{abstract}
\setcounter{tocdepth}{2}

\tableofcontents

\section{Introduction}
In this paper, we consider the following linear partial integro-differential equation 
\begin{equation}\label{DFFP}\tag{DFFP}
    \partial_t u(t,x)=\frac{1}{\e^{2s}}[J^s_\e*u-u]+\div(x u)
\end{equation}
where $ u = u(t,x) $ is a time-dependent probability density on $\mathbb{R}^d$, $\e > 0$ is a scaling parameter, and $ s \in (1/2,1) $ is a fractional index. 
We define $J_\e^s$ by
$$
J^s_\e(x) := \e^{-d} J^s\left(\frac{x}{\e}\right).
$$

The kernel $ J^s $ is a nonnegative,  nonsingular, heavy-tailed probability distribution satisfying
$$
J^s(x) \sim \frac{C}{|x|^{d+2s}} \quad \text{as } |x| \to \infty.
$$
The first term on the right-hand side of \eqref{DFFP}, namely the nonlocal operator
\begin{equation}\label{eq: diffoper}
   \mathcal{A}^s_\e u= \frac{1}{\e^{2s}} \bigl[J^s_\e * u - u\bigr]=\frac{1}{\e^{2s}}\ird J^s_\e(y)[u(x-y)-u(x)]\d y 
\end{equation}
models a Lévy-type diffusion process driven by the kernel $J^s$.
 The second term, $\div(x u)$, corresponds to a confining drift induced by the quadratic potential $ V(x) = \frac{1}{2} |x|^2 $. This drift prevents the mass from escaping to infinity and ensures the existence of a unique probability equilibrium state.
 
It is well known that, under correct normalisation of the kernel $J^s$, the operator $\mathcal{A}_\e^s$ converges (in a suitable sense) to the fractional Laplacian $-(-\Delta)^s u$ as $\e \to 0$---see for instance \citep{AndreuVaillo2010}. As a result, one expects that the solution $ u^\e$  of \eqref{DFFP} converges, as $\e \to 0$, to the solution  $v^s $ of the fractional Fokker–Planck equation:
\begin{equation}\label{eq: FFP}\tag{FFP}
    \partial_t v^s = -(-\Delta)^s v^s + \div(x v^s).
\end{equation}
We  point out that the local drift term naturally arises when starting from the PDE
\begin{equation}\label{eq: NLDIFF}
\partial_t f = \mathcal{A}_\e^s f
\end{equation}
and carrying out the superdiffusive change of variables commonly used to transform the fractional diffusion equation into the fractional Fokker--Planck equation. Indeed, defining
$
u(t,x) := e^{d t} f\Big(\frac{e^{2 s t} - 1}{2 s}, e^{t} x\Big),
$
then $u$ solves
\begin{equation}\label{eq: NANLFP}
\partial_t u=\frac{1}{\e(t)^{2 s}} (J_{\e(t)} * u - u)+\div(xu),
\end{equation}
where the time-dependent scaling parameter is $\e(t) := e^{-t}$.  
By ``freezing'' the time dependence in $\e(t)$, we recover the equation \eqref{DFFP} studied in this paper.

 We  remark that  \eqref{DFFP} is a natural generalisation of the equation studied by the authors in \citep{canizotassi2024}, namely
\begin{equation}
    \label{eq: NLFP}
    \tag{NLFP}
    \partial_t u = \frac{1}{\e^2} (J*u-u) +\div(xu)
\end{equation}
where the kernel $J$ was a centred probability distribution with finite second moment, approximating uniformly in time the classical Fokker--Planck equation
\begin{equation}
    \label{eq:FP}
    \tag{FP}
    \partial_t v = \Delta v + \div(xv)
\end{equation}
This paper is the natural follow-up to our previous work, where  the kernel is now heavy-tailed  $J^s\sim|x|^{-d - 2s}$, and consequently does not have second moment.

Equations similar to \eqref{DFFP} and \eqref{eq: NLFP} often arise  in various applications, e.g. in biological models and models of collective behaviour (see below for references). Furthermore, our interest in these equations lies in developing strategies to obtain uniform limits as highlighted in \citep{canizotassi2024}, based on Harris's Theorem.

A key challenge in the analysis of nonlocal equations is to ensure that the large-time behaviour is preserved in the limit of the parameters, here $\e\to0$ and $s\to 1$. Constants arising in entropy-based estimates often blow up or vanish as one changes the parameters, making them difficult to obtain uniform estimates. 
 Moreover, since the kernel is not singular, the equation lacks regularizing effects, which makes the analysis more challenging.
 By employing a Harris-type approach we are able to overcome these difficulties and to prove exponentially fast convergence towards the equilibrium, with a rate that is independent of $\e$ and $s$.

The uniformity in $\e$ of the rate of convergence also provides insight on the asymptotic behaviour of nonautonomous equation obtained via a self-similar change of variable as in \eqref{eq: NANLFP}.

These uniform results underline a robust structure of the equation, showing that the ``approximation'' does not get lost in the long--time. This stability has significant implications in numerical analysis, where discretisation schemes can be viewed as nonlocal approximations of PDEs and face similar challenges. The question of whether numerical methods preserve the large-time behaviour of their limiting equations has practical importance and is the subject of active research.  We refer to this property as schemes being time asymptotic-preserving. While we do not address this question directly, our results are closely related and provide insight into the underlying structure of the approximation.

Another novelty of this paper is the derivation of an explicit estimate for the \textit{long--range to short range} limit, i.e. when $s\to 1$. Such limits are relatively rare; one example is the analysis of the fractional Gagliardo norm as $s\to1$ in \citep{BourgainBrezisMironescu}.
In general, few results address the \textit{continuity of convergence} with respect to the fractional index, even though one would naturally expect fractional results to recover the classical ones in this limit. For instance, our analysis also shows that the speed of convergence to the limit distribution obtained  via quantitative generalised central limit theorems remains uniform as $s\to 1$, rather than vanishing as in the classical results of \citep{goudon_junca_toscani_BE_2002}---see Remark \ref{rmk: CLTunif}.

In general, results on the commutativity of limits, as illustrated in Figure \ref{diagram}, are both mathematically and practically interesting: they justify that the order in which parameters and long-time limits are taken does not affect the final behaviour.

\paragraph{Summary of main results}

The main results of this paper concern the three fundamental asymptotic regimes of equation \eqref{DFFP}:
\begin{enumerate}
    \item \textbf{Long-time asymptotics:} The solution   of \eqref{DFFP} converges exponentially fast to the steady state $ F^\e_s(x) $ in weighted $ L^1 $ spaces, with an explicit lower bound on the rate which is independent of both $ \e $ and $ s $.
    
    \item \textbf{Non-singular-to-singular limit ($ \e \to 0 $):} For  $ s \in (1/2,1) $, the solution $ u_\e^s(t,x) $ converges to the solution of the fractional Fokker–Planck equation \eqref{eq: FFP} as $ \e \to 0 $, uniformly in time and in the parameter $s$.
    
    \item \textbf{Long--range to short--range limit ($ s \to 1 $):}  For  $ \e > 0 $, the solution $ u_\e^s(t,x) $ converges to the solution of the Nonlocal Fokker–Planck equation \eqref{eq: NLFP} as $ s \to 1 $, again uniformly in time and in the parameter $\e$. 
\end{enumerate}
In particular, the convergence of the steady states holds.
This can be summarised by the diagram in Figure \ref{diagram}.

\begin{figure}[h]
    \centering
    \includegraphics[width=1\textwidth]{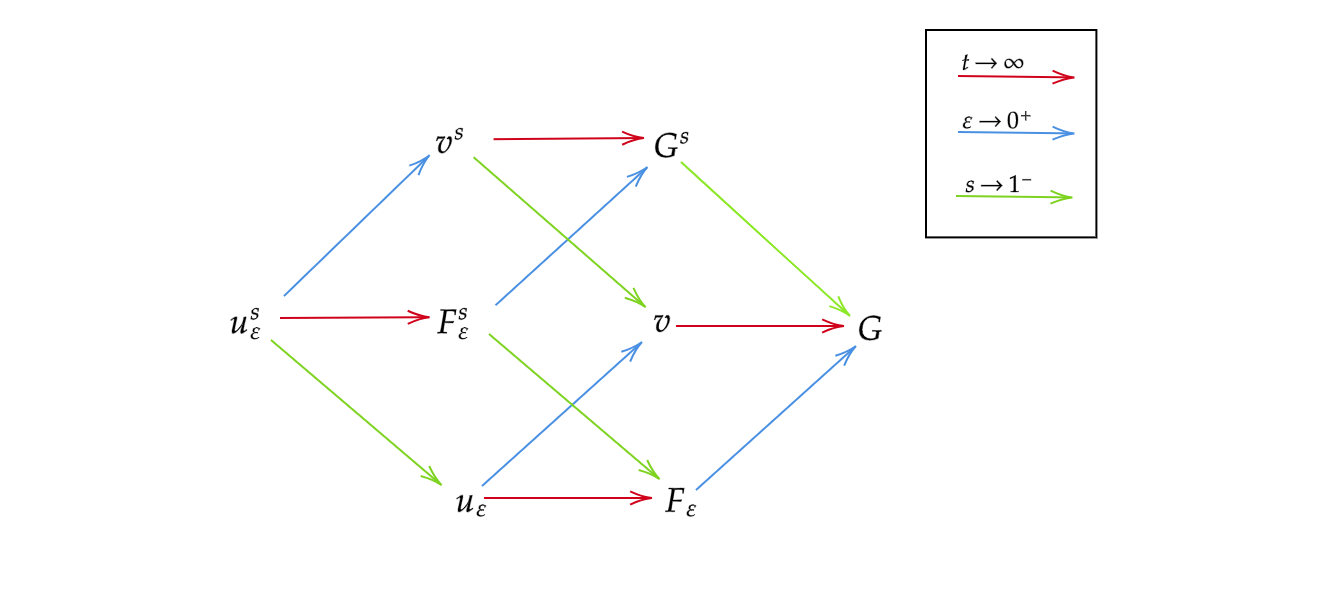}
    \caption{Quantitative convergence in $L^1_k$. The diagram illustrates that limits are uniform and independent of the specific convergence path taken.}
    \label{diagram}
\end{figure}
\bigskip

Let us state our result more precisely.  In this article, we restrict our attention to the case $s \in (1/2, 1)$.
Similar results regarding the time-asymptotic behaviour and the $\e$-limit can be obtained \textit{mutatis mutandis} for $s \in [\bar{s}, 1/2]$ (there has to be a lower bound  to avoid pathological behaviour around $s=0$).

For $k\ge 0$, we denote by $\ap{x}^k:=(1+|x|^2)^{k/2}$ and $ L^1_k(\mathbb{R}^d) $ the weighted Lebesgue space
$$
L^1_k := \left\{f :\mathbb{R}^d \to\mathbb{R} \middle|\norm{f}_{L^1_k} := \int_{\mathbb{R}^d} \ap{x}^{k} |f(x)|  dx < \infty \right\},
\quad \text{where }\ap{x} := (1 + |x|^2)^{1/2}.
$$

Similarly, given a  normed space $(S(\R^d),\norm{\cdot}_{S})$ we denote by $S_k(\R^d)$ the weighted space $$S_k:=\left\{f:\R^d\to\R: \norm{f}_{S_k}:=\norm{f   \ap{\cdot}^k}_S< \infty\right\}$$.

Standard conditions on the kernel $J^s$ in equation \eqref{DFFP}, as presented for instance in \citep{AndreuVaillo2010}, typically take the following form:
\setcounter{hyp}{-1}
\begin{hyp}\label{assJ}
    Let $s \in (0,1)$ be a (fixed) fractional index and let $J^s$ be a probability density. We assume that  there exists $\delta\in(0,2-2s]$ such that the Fourier transform of $J^s$ satisfies
    \begin{equation}\label{eq: assJ}
    \hat{J}^s(\xi) = 1 - |\xi|^{2s} + O(|\xi|^{2s+\delta}) \quad \text{as } |\xi| \to 0.
    \end{equation}
\end{hyp}

In particular, this implies that $\ird J^s(x)  dx = 1$ and  $\ird  J^s(x) x  \d x = 0$, if $s>1/2$. This condition is sufficient if the goal is only to study the non-singular-to-singular limit $\e\to 0$, for a fixed value of $s$. 
However, to obtain results that hold uniformly in the limit $s \to 1^{-}$, we must strengthen this assumption to a uniform--in--$s$ 
hypothesis.
\begin{hyp}\label{ass: sunif}
Let $(J^s)_{s\in(\frac{1}{2},1)}$ be a family of probability densities $J^s\in L^1\cap L^p$, for some $p\in(1,\infty]$. We assume the following
\begin{enumerate}[(i)]
\item \emph{(Integrability.)}  $J^s$ is the density of a Lévy measure, i.e. \begin{equation}\label{eq: levydensity}
\ird\min\{1,|y|^2\} J^s(y)\d y<\infty,
\end{equation}
\item \emph{(Uniform--in--$s$ low frequency behaviour.)} There exist constants $\delta > 0$ and $C_0 > 0$, independent of $s$, such that the Fourier transform of $J^s$ satisfies:
\begin{equation}\label{eq: expansion_uniform}
\hat{J}^s(\xi)=1-|\xi|^{2s}+R_s(\xi)\qquad\text{and}\qquad|R_s(\xi)|\le C_0|\xi|^{2s+\delta}\quad \text{ for } |\xi| \le 1.
\end{equation}

\item \emph{(Pointwise tail control.)} There exist constants $C, R>0$, and a function $\Psi\in L^1_2$ such that $\ird x \Psi(x)\d x=0$, all independent of $s$, such that 
\begin{equation}\label{eq: tailpointwise}
    J^s(x)\le C(1-s)|x|^{-d-2s}+\Psi(x)\qquad \text{for all }|x|\ge R.
\end{equation}
\end{enumerate}
\end{hyp}

To obtain the anomalous-to-classical diffusion quantitative limit, we also need to quantify in which sense and how fast the family $J^s$ is converging to $J=J^1$. We decided to present the following condition although other choices are possible.

\begin{hyp}\label{ass: remainderest}
The family $(J^s)$ converges,as $s\to1$, to a  standard nonlocal diffusion kernel $J^1$ also satisfying \eqref{eq: expansion_uniform}.
The convergence takes place in the following sense: there exists a constant $\mathfrak{C} > 0$ such that for all $s \in (1/2, 1)$ and $\xi \in \mathbb{R}^d$:$$\left| \frac{1 - \hat{J}^s(\xi)}{|\xi|^{2s}} - \frac{1 - \hat{J}^1(\xi)}{|\xi|^2} \right| \le \mathfrak{C} (1-s) |\xi|^\delta \max(1, |\log |\xi||).$$
\end{hyp}

\begin{osservazione}

The crucial difference of Hypothesis \ref{ass: sunif} is the fact that we take $\delta$ independent of $s$. 
In fact, the expansion of a general heavy tailed function is of the type
\begin{equation}\label{eq: notunifs}
1-a|\xi|^{2s}+b|\xi|^2+\dots
\end{equation}
i.e. we would have had to take $\delta=2(1-s)$. Under Hypothesis \ref{ass: sunif}, the family $(J^s)$ enjoys cancellation properties that allows us to obtain the uniform--in--$s$ limit. We refer to the Section \ref{sec: examples} for some examples of kernels that fulfil such hypotheses.
Let us remark again that if we had the expansion in \eqref{eq: notunifs}, i.e. Hypothesis \eqref{assJ}, we can prove the results independently on $\e$ but the rate of convergence would depend on $(1-s)$ and therefore would degenerate in the limit  $s\to 1$.

Let us also observe that while the Fourier expansion \eqref{eq: expansion_uniform} suggests the tail behaviour \eqref{eq: tailpointwise}, we state the pointwise bound explicitly to avoid technical Tauberian requirements (such as being monotonic for large $|x|$ or being a probability density of an infinitely divisible law) and to simplify the analysis. 
\end{osservazione}

We define the \textit{equilibrium} of equation \eqref{DFFP} as the function $F_\e^s\in L^1$ that satisfies  in the weak sense the stationary PDE
$$
\frac{1}{\e^{2s}}(J_\e^s*F_\e^s-F_\e^s)+\div(x F_\e^s)=0
$$

The main results are the following.
\begin{teorema}[Long-time asymptotics]\label{thm: main1}
Let  $(J^s)$ be a family of kernels satisfying Hypothesis \ref{ass: sunif}. For any $\e  \in (0,1)$, $s\in(1/2,1)$ and $k \in (0, 1]$:
\begin{enumerate}[(i)]\item There exists a unique probability equilibrium $F^s_\e  \in L^1_k(\mathbb{R}^d)$ of equation \eqref{DFFP}. \item For any initial data $u_0 \in L^1_k(\mathbb{R}^d)$, there exist constants $C, \lambda > 0$ such that the solution $u(t)$ satisfies:$$\|u(t, \cdot) - F^s_\e \|_{L^1_k} \le C e^{-\lambda t} \|u_0 - F^s_\e \|_{L^1_k}, \quad \forall t \ge 0.$$
    \item The constants $C$ and $\lambda$ can be chosen independently of $s$ and $\e$.

\end{enumerate}\end{teorema}
\begin{osservazione}
        If one is only interested in uniformity with respect to $\e$, a hypothesis such as Hypothesis \ref{assJ} is more suitable, and the same result holds for $k \in (0,2s)$. Furthermore, the rate of convergence could be slightly sharpen, see  Remark \ref{rmk: smalldelta}.
\end{osservazione}

This result has interesting consequences regarding the limits of the scaling parameter and the fractional index: the limits of the parameters that do not hold exclusively in compact time sets $[0,T].$  We refer to these properties as time-asymptotic preserving limits

This is summarised in the two following theorems.
\begin{teorema}[Non-singular to singular limit $\e  \to 0$]\label{thm: main2}
Let  $(J^s)$ be a family of kernels satisfying Hypothesis \ref{ass: sunif}. For $\e\in(0,1], $ $s\in(1/2,1)$ and $k\in(0,1),$ let $u^s_\e (t)$ be the solution to \eqref{DFFP} and $v^s(t)$ be the solution to the fractional Fokker--Planck equation \eqref{eq: FFP} with the same initial data $u_0 \in H^\eta\cap L^1_M(\mathbb{R}^d)$ where $\eta>2+\delta$, and $1\ge M>m > k$. Then there exists a constant $C > 0$, independent of $\e $ and $t$, such that:$$\|u^s_\e (t) - v^s(t)\|_{L^1_k} \le C \e ^{\esp}$$ where 
$\theta = \frac{m-k}{m+d/2}$ is an interpolation exponent.
Obviously taking, $\gamma=\theta\delta$, the constant is uniform in $s$.

\end{teorema}

\begin{teorema}[Long--range to short--range limit]\label{thm: main3}Let  $(J^s)$ be a family of kernels satisfying  Hypotheses \ref{ass: sunif} and \ref{ass: remainderest}.  Let $\e\in(0,\e_{\max}], $ $s\in(1/2,1)$ and $k\in(0,1),$ where $\e_{\max}>0$   depends only on the dimension $d$.
Let $u_\e ^s$ denote the solution to \eqref{DFFP} and $u_\e^1$ the solution to \eqref{eq: NLFP} with the same initial data $u_0 \in H^\eta \cap L^1_M$ with $\eta > 2+\delta$ and $1\ge M>m > k$.

Then, there exists a constant $C$ independent of $\e,s$ such that for all $t>0$ 
$$ \norm{u_\e ^s(t) - u_\e ^1(t)}_{L^1_k} \le C (1-s)^\theta \left(1 + \e ^{\theta \delta} |\log\e |\right),$$ where $\theta = \frac{m-k}{m + d/2}$ is an interpolation exponent.
\end{teorema}

\paragraph{Previous literature and related results}
The present work naturally extends \citep{canizotassi2024}, which focused on kernels with finite second moments. An equation closely related to \eqref{DFFP} equation, called  Discrete Fractional Fokker-Planck equation, is investigated in \citep{mischler_uniform_2017} as an approximation to equation \eqref{eq: FFP}. However, their kernel $J^s$ is not genuinely heavy-tailed, being instead a truncated version that exhibits heavy-tail behaviour only in the $\e$ limit. In contrast, our work employs the fractional diffusion operator approximation developed in \citep{AndreuVaillo2010}, which features an authentically heavy-tailed convolution kernel. 

There are several results on
similar equations, of which we give a short summary now, but we
highlight that nonlocal-to-local limits which preserve the asymptotic
behaviour in the limit are comparatively rare, and are a current source
of interest in several contexts.

As already stated, the approximation of the fractional heat equation  PDE $\partial_t u = J^s*u - u$ is studied in \citep{AndreuVaillo2010}. For $s=1$, the same PDE is studied in
\citep{Rey2013}, with a focus on its large-time behaviour. In this case
the motivation comes from both physics (since this PDE is thought to
be a more appropriate diffusion model in some cases) and numerical
analysis.
Various  variants have been studied, including
 nonlinear versions \citep{ignat2007nonlocal}, models for collective behaviour \citep{auricchio2023}, and gene expression  \citep{cai2006stochastic}, with the latter analysed via entropy methods in \citep{Canizo2018b}. 
For what it concerns the asymptotic stability of the numerical schemes we refer to 
\citep{ayi_structure-preserving_2022, dujardin2020coercivity, cances2020large, bessemoulin2020hypocoercivity}.

 Similar questions also arise in kinetic theory where various techniques for studying asymptotic behaviour have emerged. This includes spectral methods \citep{Gualdani2018}, hypocoercivity techniques \citep{villanihypo, dolbeauthypo}, Harris's theorem \citep{canizo2020hypocoercivity}, and entropy methods \citep{lods2008relaxation, Carrillo2007, bisi2015entropy}. We emphasise that Harris-based techniques appear particularly well-suited for establishing correct asymptotic behaviour in parameter limit regimes.

\paragraph{Plan of the paper}
In section \ref{sec: preliminar} we first gather some well established results on the fractional Laplacian and in general on the fractional diffusion equation. We also define and state the main properties of the stable laws, which are a fundamental object of our study. After that we provide three examples of families of kernels that fulfil the hypotheses.
In the second part of this section we turn our attention to the well-posedness of equation \eqref{DFFP} and we give two representations of its solution, one using Fourier transform and the other via Wild's sums which will be crucial in our asymptotic analysis.
In the last part of the section we prove the uniform--in--$s$ generalised Berry-Esseen central limit theorem.

Section \ref{sec: asymptotic} is the core of the paper, where we prove Theorem \ref{thm: main1}.
We start by giving a brief introduction on Harris's theorem which requires two conditions: Lyapunov confinement condition and a uniform positivity condition. To get a result which is stable under the scaling parameters, those conditions have to be proven uniformly in the parameter $s$ and $\e$.
The first one  is contained in section \ref{subsec: lyap}.
The technique is standard, although extra care is needed to ensure that the estimates remain uniform in $s$.
To prove the second condition, we follow the same strategy employed in our previous paper \citep{canizotassi2024}: using the Berry-Esseen theorem we establish a uniform lower bound for the solution expressed in the Wild sums formulation. 
In \ref{subsec: shape}, we briefly give some results on the regularity of the equilibrium.

In section \ref{sec: epsilon}, we investigate the non--singular--to--singular limit ($\e\to0$), of the solution of \ref{DFFP}, providing the proof of Theorem \ref{thm: main2}. This limit corresponds, in the case $s = 1$, to the  convergence from a nonlocal to a local regime studied in \citep[Sec.4]{canizotassi2024}. 

The strategy to achieve this is as follows:

\begin{enumerate}[(i)]
    \item \textbf{Consistency of the Generator.} We begin by establishing the  consistency of the operator $A_\e^s$ with the fractional Laplacian $-(-\Delta)^s$, formalised in Proposition \ref{prop: consistencyEPS}.
    This means that, when applied to sufficiently regular 
    test functions, the generator of the equation \eqref{DFFP} converges to that of the fractional Fokker--Planck equation \eqref{eq: FFP}.
    It is natural to analyse this rates of convergence in the $L^2$ space. However, since our main goal involves convergence in the weighted $L^1_k$ space, we will make use of interpolation inequalities to transfer the result. This step inevitably leads to a loss in the sharpness of the rate of convergence.

    \item \textbf{Stability Estimate and Finite-Time Convergence.} A key ingredient in the analysis is a stability estimate, given in Proposition \ref{prop: stabilityEPS}.
     This result ensures that the convergence of the operators implies convergence of the associated solutions over finite time intervals---Theorem~\ref{thm: convfinitetimeEPS}.

    \item \textbf{Convergence of Steady States.} Thanks to the existence of a spectral gap, and by Hille--Yosida theorem, we are able to transfer the operator-level convergence to the level of stationary solutions. This result, which quantifies the convergence of equilibrium states as $\e \to 0$, is stated in Theorem~\ref{thm: conveqEPS}.

    \item \textbf{Uniform--in--Time Convergence.} Finally, the combination of spectral gaps and convergence of the steady states leads to convergence also for large times ---stated in Theorem \ref{thm: convunifEPS}.
    This result, along with the already established convergence for finite times gives  \emph{uniform--in--time convergence} of solutions. 
\end{enumerate}

In section \ref{sec: s} we give the proof to Theorem \ref{thm: main3}, namely the behaviour of the solution of the equation as $s\to 1$ and quantitative convergence estimates to the solution of equation \eqref{eq: NLFP}.
The strategy  mirrors  the one employed in the previous section.

Again, as in the previous limit, the consistency result is much more natural in $L^2$. This is reflected into a non-optimal rate of convergence due to interpolation.


\section{Preliminaries}\label{sec: preliminar}
Throughout this paper, the Fourier transform of $f$ is denoted by either  $\mathcal{F}[f]$ or $\hat{f}$ and defined as
$$
\hat{f}(\xi)=\ird e^{-ix\cdot \xi} f(x)\d x 
$$
\subsection{Fractional Calculus}
Let us recall some useful definitions and results on fractional calculus. 
\paragraph{Fractional Laplacian}
    In  recent years, the fractional Laplacian has received much attention because it provides a rich and interesting framework  for addressing nonlocal diffusion problems. It would be impossible to give a complete bibliography but we will give an overview for a reader unfamiliar with these  problems.

\begin{definizione}    
For any $u\in \mathcal{S}$, the Schwartz space $\mathcal{S}$ of rapidly decreasing functions, the fractional operator $(-\Delta)^s$, is defined  as
    \begin{equation}\label{def: FLI}
        -(-\Delta)^su(x)=C_{d,s}\pv\ird \frac{u(y)-u(x)}{|x-y|^{d+2s}}\d y
    \end{equation}
    where the integral is to be understood in the principal value sense, and  $C_{d,s}$ is a normalizing dimensional constant
    $$C_{d,s}=\frac{2^{2s}\Gamma\left(\frac{d}{2}+s\right)}{\pi^{d/2}\abs{\Gamma(-s)}}.$$
     Observe that if $s\in(0,1/2)$, we don't need to take the principal value since the singularity at the origin cancels out and the integral converges.
     This definition can be extended to the space of tempered distributions $\mathcal{S}'$, see \citep{di_nezza_hitchhikers_2012}.
\end{definizione}

    Different fractional operators have been considered in the literature, arising from different problems---see e.g. \citep{caputo1967linear}, \citep{oberhettinger2013tabellen} and \citep{servadei2014spectrum}.
    
     We also define, for sake of completeness, the fractional Sobolev spaces $W^{s,p}$, a natural extension of Sobolev spaces.
     \begin{equation}\label{def: fracsob}
         W^{s,p}(\Omega):=\Big\{u\in L^p(\Omega):  [u]_{W^{s,p}}:=\left(\int\int_{\Omega\times\Omega} \frac{|u(x)-u(y)|^p}{|x-y|^{d+ps}}\d x\d y\right)^{1/p}<\infty
         \Big\}.
     \end{equation}
     For $s\in(0,1)$, $ W^{s,p}$ is an intermediate Banach space between $L^p$ and $W^{1,p}$ when endowed with its natural norm
     \begin{equation}\label{def: fracsobnorm}
      \norm{u}^p_{W^{s,p}(\Omega)}:=\int_\Omega |u|^p\d x+[u]_{W^{s,p}}^p.
     \end{equation}
     $[u]_{W^{s,p}}$ is sometimes called Gagliardo seminorm.
     
     It has been proved in 
     \citep{BourgainBrezisMironescu}
     that 
     $$\lim_{s\to 1^-}(1-s)[u]_{W^{s,p}}^p=C_1\ird|\nabla u|^p= C_1[u]^p_{W^{1,p}}$$
     for a suitable constant $C_1$ depending only on $d$ and $p$.     
  When $p=2$, then  $W^{s,2}(\Omega):=H^{s}(\Omega)$ turns out to be a Hilbert space for every $s\in(0,1)$.

     Another equivalent formulation of \eqref{def: FLI} in terms of Fourier transform is the following
\begin{definizione}    
        Let $(-\Delta)^s:\mathscr{S}\to L^2$ be the fractional Laplacian operator. Then for every $u\in \mathscr{S}$,
    \begin{equation*}
        (-\Delta)^s u=\mathcal{F}^{-1}(|\xi|^{2s} \mathcal{F}(u))\qquad s\in(0,1)
    \end{equation*}

\end{definizione}

    Moreover, $H^{s}(\R^d)$ coincides with $$\hat{H}^{s}:=\Big\{u\in L^2(\R^d):\ird(1+|\xi|^{2s})|\hat{u}(\xi)|^2\d\xi<\infty\Big\}$$ 
     We refer to \citep{di_nezza_hitchhikers_2012} and references therein for a survey on these spaces as well as on the fractional Laplacian operator. 

   The fractional Laplacian recovers the classical Laplacian in the limit $s \to 1^-$. It is well-known that $(-\Delta)^s \phi \to -\Delta \phi$ pointwise for $\phi \in C_c^2(\mathbb{R}^d)$. Our analysis requires, however, a quantitative estimate in $L^2$. Using the Fourier representation, one can show that for $\phi \in H^\eta(\mathbb{R}^d)$ with $\eta > 2$ there exists a constant $C > 0$ independent of $s$ such that:
\begin{equation}\label{eq: consistency-FL}
\norm{(-\Delta)^s \phi - (-\Delta) \phi }_{L^2} \le C(1-s) \norm{\phi}_{H^\eta}.
\end{equation}

     \paragraph{Fractional diffusion and Fractional Fokker-Planck}
     
  The basic parabolic equation involving the fractional Laplacian is the fractional diffusion equation, sometimes called the fractional heat equation:
$$
\partial_t f = -(-\Delta)^s f.
$$

An interesting parallel between the classical heat equation and its fractional counterpart is highlighted, for instance, in \citep{bucurvaldinoci2016}, from a probabilistic point of view. Indeed, the fractional heat equation can be seen as the scaling limit of a stochastic process where particles move randomly and can perform arbitrarily long jumps, in contrast to the local behaviour of Brownian motion.

Performing the change of variables already mentioned in the introduction,
$$
v^s(t,x) = e^{dt} f\left( \frac{e^{2st} - 1}{2s}, e^t x \right),
$$
it is easy to verify that $v^s$ solves the fractional Fokker-Planck equation \eqref{eq: FFP}.

We now introduce the following definitions.
\begin{definizione}\label{def_ infinitDiv}
A probability measure $\rho$ is {infinitely divisible} if, for all $n \in \mathbb{N}$, there exists another probability measure $\rho_n$ such that 
\begin{equation*}
\rho = \underbrace{\rho_n*\dots*\rho_n}_{n \text{ times}}
\end{equation*}
\end{definizione}
A well-known and important example of infinitely divisible distributions are the stable laws, and as a particular case, the symmetric ones.
\begin{definizione}\label{def: stable}
For $s \in (0,1]$, we denote by $G^s$ the standard symmetric stable law of order $2s$, whose Fourier transform is given by
$$
\widehat{G}^s(\xi) = e^{-|\xi|^{2s}}.
$$
To represent stable laws with arbitrary scale, we introduce the notation $G^{s;\gamma}(x)$, whose Fourier transform is
$$
\widehat{G}^{s;\gamma}(\xi) = e^{-\gamma |\xi|^{2s}}.
$$
\end{definizione}

In particular, if $\gamma = s= 1$, $G^1$  is a Gaussian distribution with covariance matrix $2\Id$, where $\Id$ is the $d\times d$ identity  matrix. We  write $G := G^1$ when no ambiguity arises.

Stable laws have many interesting properties---see for example \citep{nolan_stable_2020}. The most relevant ones to us are the following:

\begin{proposizione}\label{prop: propSL}
The family of symmetric $2s$‑stable laws $G^{s;\gamma}$ defined by
$$
\widehat{G}^{s;\gamma}(\xi)=e^{-\gamma\,|\xi|^{2s}}, 
\quad s\in(0,1], \gamma>0,
$$
enjoys the following properties:
\begin{enumerate}[(i)]
\item\emph{Scaling.}
  $$
    G^{s;\gamma}(x)=  G_{\gamma^{1/2s}}^s(x)
  $$
  \item \emph{Sum of stable random variables:} If two independent random variables have density following a $2s$ stable distribution, then also their sum will follow a $2s$ stable distribution. Moreover, the stable distributions are the only ones with this property.
   As a consequence, one has
  $$
    G^{s;\gamma_1} * G^{s;\gamma_2}
    = 
    G^{s;\gamma_1+\gamma_2}, 
    \quad \forall\,\gamma_1,\gamma_2>0.
  $$
  \item Convergence in the fractional index: it can be proved that 
  $$
  G^s\to G\qquad s\to 1
  $$
  in several senses. In particular, we have the following quantitative convergence result in $L^1_k$: for some $\alpha>0$
  \begin{equation}\label{eq: stabletogauss}
      \norm{G^s-G}_{L^1_k}\le C(1-s)^\alpha.
  \end{equation}

  \item \emph{Heavy-tail asymptotics:}
  $$
    G^s(x) \sim \frac{K_{d,s{}}}{|x|^{d + 2s}}\qquad \text{ as } |x| \to \infty
  $$
  where $$K_{d,s}= \frac{s2^{2s}\Gamma(d/2+s)\sin(s\pi)}{\pi^{d/2+1}} $$
 and when $s\sim1$ then $K_{d,s}\sim (1-s)$. This  behaviour of $G^s$ provides can be also seen in \eqref{eq: tailpointwise}.
 
  \item \emph{Radial symmetry and positivity:}
  each $G^{s;\gamma}(x)$ is a nonnegative, radially symmetric, smooth probability density on $\R^d$.
\end{enumerate}
\end{proposizione}
Similar to the standard Fokker-Planck, the following theorem holds.
\begin{teorema}\label{thm: propFFP}
	Let $v_0\in L^1$. Then
    $$
    v^s(t, x) = (v_0 \circ e^t) * G^{s,\sigma(t)}(x)$$
where $\sigma(t) = \frac{1-e^{-2st}}{2s}$ and $\circ$ denotes the composition. In particular, $H^\infty$ regularity and strict positivity  follow from the properties of the stable laws.
\end{teorema}
Speed of convergence towards the equilibrium has been studied via representation in \cite{vazquez2017asymptoticbehaviourfractionalheat}, via entropy methods in \citep{gentil_levy-fokker-planck_2006}, \citep{bilerKarch_2003_generalizedFP}, and \citep{lafleche_fractional_2020},  or via operator splitting in \citep{tristani2013fractional}. In particular for the latter, the author proved
\begin{equation}\label{eq: SG-FFP}
    \norm{v^s(t,\cdot)-\tilde{G}^s}_{L^1_k}\le Ce^{-\lambda t}\norm{u_0-\tilde{G}^s}_{L^1_k}
    \end{equation}
 
Combining \eqref{eq: SG-FFP} with an argument of consistency of the operator \eqref{eq: consistency-FL} and stability, plus an interpolation inequality, in the spirit of Section \ref{sec: s}, one obtains uniform--in--time convergence.
\begin{teorema}\label{thm: unifconvFL}
Let $1\ge M>m>k$ and let  $v^s$ and $v$ denote the solutions of equations \eqref{eq: FFP} and \eqref{eq:FP}, respectively, with the same initial data $v_0\in H^{\eta}\cap L^1_M$, for $\eta>2$. Then, there exists a constant $C$ and $\alpha>0$, both independent of $s$, such that  
    \begin{equation*}
    \norm{v^s(t)-v(t)}_{L^1_k}\le C(1-s)^\alpha\qquad \forall t\in[0,\infty).
\end{equation*}
\end{teorema}

\begin{osservazione}
    Theorem \ref{thm: unifconvFL} and Equations \eqref{eq: SG-FFP} and \eqref{eq: stabletogauss} are three of the arrows in Figure \ref{diagram}.
\end{osservazione}

\subsection{Examples of kernels}\label{sec: examples} 

\begin{enumerate}
\item A standard example of a family satisfying Hypotheses \ref{ass: sunif} and \ref{ass: remainderest} is given by the family of stable laws, namely $$\hat{G}^s(x)=e^{-|\xi|^{2s}},$$ 
as shown in the following lemma.
\begin{lemma}\label{lemma: Ass2Stable}
The family $J^s=G^s$ satisfies Hypotheses \ref{ass: sunif} and \ref{ass: remainderest}, with $J^1=G$.
\end{lemma}
\begin{proof}
By a straightforward Taylor expansion, and since $s\in(1/2,1)$, Hypothesis \ref{ass: sunif} is satisfied with $\delta=1$.

		Define the function $\phi:\R^+\to \R^+$, 
		$$\phi(u)=\frac{e^{-u}-1}{u},$$
		and observe that $\phi$ is Lipschitz of constant $1/2$. Therefore
		\begin{multline*}				\abs{\frac{\hat{J}^s(\xi)-1}{|\xi|^{2s}}-\frac{\hat{J}(\xi)-1}{|\xi|^2}}=|\phi(|\xi|^{2s})-\phi(|\xi|^2)|\le \frac{1}{2}\abs{|\xi|^{2s}-|\xi|^2}\\
			=\frac{1}{2}|\xi|^{2s}|1-|\xi|^{2(1-s)}|\le (1-s)|\xi|^{2s}\abs{\log(|\xi|)}(1+|\xi|^{2(1-s)}),\end{multline*}
		since 
		$$
		|1-|x|^a|\le a\abs{\log(|x|)}(1+|x|^a)
		$$		
\end{proof}
\item 
Another family satisfying the hypotheses is given by the fundamental solution of the fractional screened Poisson equation,
\begin{equation}\label{eq: screenedpoiss}
[(-\Delta)^s+\lambda^2]J^s=f
\end{equation}
with $\lambda=1$ and $f=\delta_0$. The solution can be written as 
$$
\hat{J}^s(\xi)=\frac{1}{1+|\xi|^{2s}}=\sum_{n\ge0}(-1)^n|\xi|^{n2s}
$$
$J^s$ converges to the fundamental solution of the screened Poisson equation  $J^1$ whose Fourier transform is $\frac{1}{1+|\xi|^2}$.
For example  $J^1(x)=\frac{1}{2}e^{-|x|}$ in $1$ dimension.
\medskip

\item 
In general, a heavy tailed kernel satisfying
$$J^s(x)\sim |x|^{-d-2s}$$ 
fulfils Hypothesis \ref{assJ} but not \ref{ass: sunif}.

In order to recover \ref{ass: sunif}, one must suitably tune the constants and add a rapidly decaying ``core'' function in order to cancel the second-order term in the Fourier expansion, as shown in the following example.
\begin{lemma}
    Let $\f_1(x)$ be a Student's t--distribution   of parameter $\nu=2s$.
    $$
    \f_1(x)=\frac{\Gamma\left(s+\frac{d}{2}\right)}{\Gamma(s)(2s\pi)^{d/2}}\left(1+\frac{|x|^2}{2s}\right)^{-\left(\frac{d}{2}+s\right)}.
    $$    
    Define $$a(s)=-\frac{\Gamma(s)2^{2s}}{\Gamma(-s)(2s)^s}>0.$$    
    Let      $\f_2$ be a Gaussian with variance $2\gamma$ where $$\gamma=\frac{1}{2}\frac{a(s)}{1-s}\frac{s}{1-a(s)}.$$
    Observe that $a(s)\sim 2(1-s)$  and $\gamma\sim 1$, when $s\sim 1$ 

    Then the function 
     \begin{equation}
        \label{eq: sharpApprox}
    J^s(x)=a(s)\f_1(x)+(1-a(s))\f_2(x)
     \end{equation}
    satisfies Hypotheses \ref{ass: sunif}  and \ref{ass: remainderest}.
\end{lemma}

\begin{proof}
Let's write $J^s(x)=c_1\f (x)+c_2\f_2(x)$.
The Fourier transform  of $\f_1$ is known to be
$$
\hat{f}_1(\xi)=\frac{2^{1-s}}{\Gamma(s)}(\sqrt{2s}|\xi|)^sK_s(\sqrt{2s}|\xi|)
$$
where $K_s$ is the modified Bessel function of the second kind which is defined as 
$$
K_\nu(x)=\frac{\pi}{2}\frac{I_{-\nu}(z)-I_\nu(z)}{\sin(\nu\pi)}
$$
and  $I_\nu$ is the modified Bessel function of the first kind which admits the series expansion
$$
I_\nu(z)=\left(\frac{z}{2}\right)^\nu\sum_{k=0}^\infty
\frac{1}{k!\Gamma(\nu+k+1)}\left(\frac{z^2}{2}\right)^k$$

 To show that Hypothesis \ref{ass: sunif} is fulfil, we expand $\hat{f_1}$, by using the properties of the $\Gamma$ function, obtaining
$$
\hat{\f}_1(\xi)=1-A(s)|\xi|^{2s}+B(s)|\xi|^2+O(|\xi|^{2s+2})
$$
with $A(s)=-\frac{\Gamma(-s)(2s)^s}{\Gamma(s)2^{2s}}>0$ and 
    $B(s)=\frac{s}{2(1-s)}>0$.

    Notice that $A(s)=-a(s)^{-1}$
The expansion of the Gaussian is 
$$
\hat{\f}(\xi)=1-\gamma|\xi|^2+O(|\xi|^4)
$$
Thus
\begin{equation*}
    \hat{J^s}(\xi)=c_1\hat{\f}_1(\xi)+c_2\f_2(\xi)=
    [c_1+c_2]+[-c_1A(s)]|\xi|^{2s}+[c_1B(s)-\gamma c_2]|\xi|^2
\end{equation*}
The three conditions to be satisfied are the following 
\begin{enumerate}[(i)]
    \item $c_1+c_2=1$
    \item $-c_1A(s)=-1$
    \item $c_1B(s)-\gamma c_2(s)=0$
\end{enumerate}
From $(ii)$ and using the asymptotic properties of the $\Gamma$ function 
we have
$$
c_1=A(s)^{-1}=a(s)\qquad\text{and } \qquad c_1\sim 2(1-s) \text{ as }s\sim 1
$$
From $(i)$, then 
$$
c_2=1-a(s)\qquad\text{and } \qquad c_2\sim 2s-1 \text{ as }s\sim  1
$$
We have to check that $c_2(s)>0$. This happens  as long $|A(s)|>1$, which is easy to prove if $s\in(1/2,1).$

Finally to cancel the term corresponding to $|\xi|^{2}$ one can  choose
$$
\gamma=\frac{a(s)}{2(1-s)}\frac{s}{1-a(s)}\qquad\text{and } \qquad \gamma\sim 1 \text{ as }s \sim  1
$$
to obtain the first claim.

By considering $\f_1$ at the next  orders of the expansion, one can show that $(J^s)$ also satisfies  Hypothesis \ref{ass: remainderest} with $J^1(x)=\lim_{s\to 1}J^s(x)=G^1(x)$ but we do not provide here the details.

\end{proof}

\begin{osservazione}
   The function defined in \eqref{eq: sharpApprox} provides a sharper approximation of both the stable law and the fundamental solution of \eqref{eq: screenedpoiss}, which do not admit explicit closed-form expressions. This construction also yields a more accurate approximation of the fractional Laplacian than the pure stable law, and more robust in the limit $s\to 1$.
\end{osservazione}
\end{enumerate}

\subsection{Well-Posedness of the Initial Value Problem}
We collect here some basic well-posedness results for the Cauchy problem with fixed $\e$: existence and uniqueness of mild solutions, as well as continuous dependence on initial data, follow from standard contraction arguments. We also provide representations of the solution in Fourier variables and in terms of Wild sums. Since the proof closely follows classical methods and is, \textit{mutatis mutandis}, the same as in \citep{canizotassi2024}, we only outline the main steps and refer to that work for full details.
Let us define the semigroup $T_t:L^1_k\to L^1_k$
\begin{equation}
  \label{eq:St}
  T_t[u_0] (x) \equiv T_t u_0(x) := e^{(d-1/\e^{2s})t} u_0(e^t x)
  \qquad \text{for $t \ge 0$, $x \in \R^d$,}
\end{equation}
solution in the sense of semigroup of the PDE
\begin{equation*}
  \p_t u = \div(xu)-\frac{1}{\e^{2s}} u, \qquad u(0,x) = u_0(x).
\end{equation*}

\begin{definizione}[Mild solution]
  \label{dfn:mild}
  Let $0\le k<2s$, and take $J$ satisfying Hypothesis
  \eqref{ass: sunif} on $J$ and $J\in L^1_k$ with an initial data
  $u_0 \in L^1_k(\R^d)$. For a fixed time $T^* \in (0,+\infty]$.  A \emph{mild
    solution} of equation \eqref{DFFP} on $[0,T^*)$ with initial
  condition $u_0$ is a function $u \in C([0,T^*), L^1_k(\R^d))$ such
  that
  \begin{equation}\label{mild1}
    u(t) = T_t u_0 +  \frac{1}{\e^{2s}}\int_0^t T_{t-s}
    \big[ J_\e*u(s)\big] \d s
    \qquad
    \text{for all $t \in I$,}
  \end{equation}
\end{definizione}
Existence and uniqueness of such solutions can be proved via standard contraction arguments.
\begin{teorema}[Well-posedness]
  \label{thm:wellposed-fractional}
  Let $k \in(0,1)$ and let $(J^s)$ be a family of kernels satisfying Hypothesis \ref{ass: sunif}. Then, for all $s\in(1/2,1)$ and for every initial condition $u_0 \in L^1_k$, the  equation \eqref{DFFP}
  $$
  \partial_t u=\frac{1}{\e^{2 s}}(J_\e*u-u) + \div(xu)
  $$
  admits a unique mild solution $u \in \mathcal{C}([0,\infty), L^1_k)$ in the sense of Definition \ref{dfn:mild}. Moreover, the solution depends continuously on the initial data: for any two mild solutions $u$ and $v$ with respective initial data $u_0$ and $v_0$, we have
  $$
  \norm{u(t)-v(t)}_{L^1_k} \le e^{\frac{t}{\e^{2 s}}(c_k-1)} \norm{u_0-v_0}_{L^1_k},
  $$
  where $c_k=2^k\norm{J}_{L^1_k}$.
\end{teorema}

\begin{proof}[Sketch of Proof]
  The argument follows the same steps as in \citep{canizotassi2024}. Define the space $\Y := \mathcal{C}([0,T^*], L^1_k)$ and the operator
  $$
  \Psi[u](t) := T_t u_0 + \frac{1}{\e^{2 s
}} \int_0^t T_{t-s}(J_\e * u(s)) \, ds,
  $$
  where $T_t$ is the semigroup associated with the transport part. Using the same weighted Young inequality and the contractivity of $T_t$ in $L^1_k$---Lemma 2.3 and 2.4 of \citep{canizotassi2024}, we obtain
  $$
  \norm{\Psi[u](t) - \Psi[v](t)}_{L^1_k} \leq \frac{c_k}{\e^{2 s}} \int_0^t \|u(s) - v(s)\|_{L^1_k} \, ds.
  $$
  Iterating $\Psi$
   \begin{equation*}
     \norm{\Psi^n[u](t)-\Psi^n[v](t)}_{\Y}
     \le
     \Big(\frac{c_kT^*}{\e^{2s}}\Big)^n\frac{1}{n!}\|u-v\|_{\Y}.
   \end{equation*}
  yields a contraction for small $T^*$ (or equivalently larger $n$), and a fixed point argument gives existence and uniqueness. The continuity estimate follows by applying Gronwall’s inequality to the difference of two solutions, using
  $$
  \|u(t) - v(t)\|_{L^1_k} \leq e^{-t/\e^2 s} \|u_0 - v_0\|_{L^1_k} + \frac{c_k}{\e^{2 s}} \int_0^t e^{-\frac{t-s}{\e^{2s}}} \|u(s) - v(s)\|_{L^1_k} \, ds.
  $$
\end{proof}
It is possible to prove existence and uniqueness for measures. Let $\M$ be the space of finite signed Radon measures and let $\norm{\cdot}_{\M_k}$ denote the weighted total variation norm. Then a mild solution exists and it is unique in the space $\mathcal{Z}=C([0,T^*], \X_R)$, where the continuity has to be intended in Bounded Lipschitz norm $\X_R:=\{\mu\in\M:||\mu||_{\M_k}\le R\}$.

\begin{osservazione}
The operator $L_\e^s$ can be decomposed as $L_\e^s = L_1 + L_2$, where $L_1 u = \e^{-2s}(J_\e * u)$ is a bounded linear operator on $L^1_k$, and $L_2u = \div(xu) - \e^{-2s}u$. Since $L_2$ generates a strongly continuous contraction semigroup $T_t$ on $L^1_k$ (given by \eqref{eq:St}), the bounded perturbation theorem for semigroups \citep{EngelNagel2001} ensures that $L_\e^s$ also generates a $C_0$-semigroup, which justifies the use of mild solutions in Definition \ref{dfn:mild}.
\end{osservazione}

\subsection{Representation of the solutions}
\paragraph{Fourier.}
We now study the equation in Fourier variables. Although this is not the main point of the paper, this representation 
allows 
 to obtain closed-form expressions for $\hat{u}(t,\xi)$ and the unique stationary state, assuming that it exists. In particular, this will be used at the end of next section to gather some information on the regularity of the equilibrium.

Assume that the initial data $u_0$ is a probability distribution in $L^1$. The solution of \eqref{DFFP}  must satisfy the explicit formula in Fourier variables:
\begin{equation}\label{eq: solfourier-fractional}
  \hat{u}(t,\xi) = \hat{u}_0(e^{-t} \xi) \exp\left( \frac{1}{\e^{2s}} \int_0^t \zeta_\e(e^{-\tau} \xi) \, \d \tau \right),
\end{equation}
where $\zeta_\e(\xi) := \hat{J}_\e(\xi) - 1$.
Moreover, there exists at most one equilibrium $F_\e \in L^1$ for this equation, and its Fourier transform must satisfy
\begin{equation}\label{eq: eqfourier-fractional}
  \hat{F}_\e(\xi) = \exp\left( \frac{1}{\e^{2s}} \int_0^\infty \zeta_\e(e^{-\tau} \xi) \, \d \tau \right).
\end{equation}

Indeed, applying the Fourier transform to the equation yields
$$
\partial_t \hat{u} = \frac{1}{\e^{2s}} \zeta_\e(\xi) \hat{u} - \xi \cdot \nabla_\xi \hat{u},
$$
which is a linear first-order PDE in $\hat{u}$ with known characteristics. The unique solution to this equation is given by \eqref{eq: solfourier-fractional}.

To characterise equilibria, note that a stationary solution must satisfy
$$
\hat{F}_\e(\xi) = \hat{F}_\e(e^{-r} \xi) \exp\left( \frac{1}{\e^{2s}} \int_0^r \zeta_\e(e^{-\tau } \xi) \, \d \tau \right), \quad \text{for all } r \geq 0.
$$
Since $\|F_\e\|_{L^1} = 1$, we have $\hat{F}_\e(0) = 1$, and taking the limit as $r \to \infty$ yields
$$
\hat{F}_\e(\xi) = \lim_{r \to \infty} \hat{F}_\e(e^{-r} \xi) \exp\left( \frac{1}{\e^{2s}} \int_0^r \zeta_\e(e^{-\tau} \xi) \, \d \tau \right) = \exp\left( \frac{1}{\e^{2s}} \int_0^\infty \zeta_\e(e^{-\tau} \xi) \, \d \tau \right).
$$

\paragraph{Wild Sums.} When dealing with linear convolution-type
operators, Wild sums---or Dyson-Phillips  series  ---are a particularly simple way to express the solution.

\begin{teorema}\label{Thm:wild}
Let $u$ be the unique solution to Equation \eqref{DFFP} with $u_0$ as initial datum. 
Then under the hypotheses of Theorem \ref{thm:wellposed-fractional}
\begin{equation}
     u(t,x)=e^{(d-\frac{1}{\e^{2s}})t}\Big[u_0(e^tx)+\sum_{n=1}^\infty \Big(\frac{1}{\e^{2s}}\Big)^n\int_0^t\int_0^{t_{n}}\dots\int_0^{t_{2}}\Jn*u_0(e^tx) \d t_1\dots \d t_n.\Big]
\end{equation}
where
\begin{equation}\label{jn}
    \Jn:=J_{\e e^{t_n}}*\cdots*J_{\e e^{t_1}}(x).
\end{equation}
with $0\le t_1\le\dots\le t_n\le t$.
\end{teorema}

\begin{proof}

The proof is based on the classical  Dyson-Phillips expansion \citep[Ch.II, Sec.6, Thm. 1.10]{EngelNagel2001}.
Let us define by $T_t$ the semigroup generated by the operator $L_2$ defined as $L_2u=\div(xu)-\e^{2s} u$ which takes the form 
$$
T_tu_0(x)=e^{-t/\e^{2s}}e^{dt}u_0(e^t x).
$$
Then this allows us to write the solution in the Wild sum form
$$
u(t)=T_tu_0+\sum_{n\ge0}
\Big(\frac{1}{\e^{2s}}\Big)^n\int_0^t\int_0^{t_n}\cdots \int_0^tT_{t-t_n}[J_\e*T_{t_{n}-t_{{n-1}}}[J_\e\dots [J_\e*T_{t_1 }u_0]\dots]] \d t_1,\dots \d t_n
$$
Applying \citet[Lemma 2.6]{canizotassi2024} that allows to exchange the operators as 
\begin{equation*}
    J_\e*(T_tf)=T_t[J_{\e e^t}*f],
\end{equation*}
so the proof is complete.
\end{proof}

\subsection{Berry-Esseen Central Limit Theorem for heavy-tailed distributions} 
Berry-Esseen-type theorems are quantitative versions of the central limit theorem (CLT), that provide explicit estimates on the error between the distribution of properly scaled sums of random variables and the standard normal distribution.

The first version of this theorem was independently proposed by Berry \citep{berry_1941} and Esseen \citep{esseen1942liapunov}: for real, centered i.i.d.\ random variables  
$X_1, \dots, X_n$ with $\Var(X_1) = \sigma^2$ and $\E(|X_1|^3) = \rho$,  
the Berry--Esseen theorem provides a bound of the form:
$$
\sup_{x \in \R} |F_n(x) - \Phi(x)| \le \frac{C \rho}{\sqrt{n} \, \sigma^3},
$$
where $\Phi$ denotes the cumulative distribution function (c.d.f.) of the standard normal distribution, $F_n$ is the c.d.f. of the normalised sum  
$\bar{S}_n = \frac{1}{\sigma \sqrt{n}} \sum_{i=1}^n X_i$, and $C$ is a universal constant that may depend on the dimension $d$.
Efforts have been made to sharpen the bound by improving the value of the constant $C$. Recent developments in this direction can be found in \citep{bentkus_2005} and \citep{raic_2019}.

Beyond results for the cumulative distribution function, several works have established Berry--Esseen-type bounds for densities. A strengthened version was proposed in \citep{lions1995strengthened}, focusing on convergence for smooth densities. This was further studied in \citep{goudon_junca_toscani_BE_2002}, \citep{carlen_entropy_2010}, \citep{mischler_kacs_2014}, and in the author's previous paper \citep{canizotassi2024}.

Classical central limit theorems assure convergence toward the Gaussian distribution, provided that the random variables have finite variance. Berry--Esseen-type theorems give a quantitative estimate on the "distance" from the limiting distribution, assuming additional control on the $2+\delta$ moment.

When the distribution does not have finite variance, the limit law is no longer Gaussian: for instance, if $f \sim |x|^{-d-2s}$, then the sum converges to a stable law of order $2s$. This is the content of the \textit{generalised central limit theorem}. In this regime, $f$ is not in $|x|^{2s}L^1$, so the usual moment condition must be replaced, typically by something similar to Hypothesis \ref{assJ}.

In \citep[Chapter 4]{goudon_junca_toscani_BE_2002}, the authors adapt their original proof for the classical case to obtain a generalised Berry--Esseen theorem. We  show a similar result, extending our proof and \citep{canizotassi2024} to derive an estimate for non-identically distributed random variables with heavy-tailed densities. Differently from previous work, we obtain a rate of convergence that is uniform in the  index $s$.

The main result of this section is the following

\begin{teorema}[$L^\infty$ Berry-Esseen GCLT]
\label{thm: BE}
 Let $(f^s)$ be a family of density function 
    satisfying Hypothesis \ref{ass: sunif}. For $s\in(1/2,1)$, define
     $$\mathfrak{f}_n(x):=(\bar{\sigma}^{2s}n)^{d/2s}f^{\sigma_1,\dots\sigma_n}(n^{1/2s}\bar \sigma x),$$
    with  $$f^{\sigma_1,\dots\sigma_n}(x):=f^s_{\sigma_1}*f^s_{\sigma_2}*\cdots*f^s_{\sigma_n}(x)$$ and  $\bar\sigma^{2s}=\frac{1}{n}\sum_{i=1}^n\sigma_i^{2s}$. 
    
    Suppose, in addition, that there exist $l,L>0$ such that 
    $$0<l\le \sigma_i \le L\quad i=1,\dots, n$$      

Then, there exist an integer $N=N(p)$ and a constant $C_{BE} $ depending on $l,L,p, d,\rho_{2s+\delta},\norm{f_s}_{L^p}$ 
\begin{equation*}
    \norm{\Jnorm-G^s}_{L^\infty}\le \frac{C_{BE}}{n^{\delta/{2s}}}\quad \text{ for every $n\ge N$}
\end{equation*}
\end{teorema}The proof follows the same structure as the one in \citep{canizotassi2024}. However, we decided to provide a self-contained proof for the sake of completeness.
For the rest of the section we will simply denote by $f=f_s$.
We first need some preliminary lemmas.
The
first one is a generalisation of the identity
$a^n-b^n=\sum_{j=1}^n (a-b)a^{j-1}b^{n-j}$, that can be easily proven
by induction:
\begin{lemma}\label{lemma: produttoria}
	For real $a_1,\dots,a_n$ and $b_1,\dots b_n$, for any $n\ge1$
	\begin{equation*}
		\prod_{j=1}^na_j-\prod_{j=1}^nb_j=\sum_{j=1}^n\Big[(a_j-b_j)\big(\prod_{k=1}^{j-1}b_k\big)\big(\prod_{k=j+1}^n a_k\big)\Big]
	\end{equation*}
\end{lemma}
\begin{lemma}\label{lemfou}
  Let $f$ be a probability distribution on $\R^d$ satisfying
  \eqref{ass: sunif}. Then
  \begin{enumerate}[(i)]
  \item  There exists $r\in (0,1)$ depending on $\delta$ such that 
    \begin{equation*}
        |\hat{f}(\xi)|\le e^{-\frac{|\xi|^{2s}}{2}}\qquad \forall \xi\in B_r
    \end{equation*}
  \item Assume additionally that $f\in L^p$, with
    $p\in(1,\infty]$. For any $r>0$ there exists\\
    $\kappa=\kappa(\rho_{2s+\delta},r, \norm{f}_{L^p})\in(0,1)$ such
    that
    $$\sup_{\abs{\xi}\ge r}|\hat{f}(\xi)|\le \kappa(r).$$
\end{enumerate}     
\end{lemma}
\begin{proof}
\begin{enumerate}[(i)]\item 
Hypothesis \ref{ass: sunif} implies that there exist a radius $r_0>0$ and a constant $C_{r_0}$ such that
$$
\abs{\hat{f}(\xi)-1+|\xi|^{2s}}\le C_{r_0}|\xi|^{2s+\delta}
$$
Choose $r<r_0$ and such that $$C_{r_0}r^{\delta}\le\frac{1}{2}\implies C_{r_0}|\xi|^{2s+\delta}\le \frac{1}{2}|\xi|^{2s}\quad\text{for }|\xi|\le r. $$ Then
$$
|\hat{f}(\xi)|\le \abs{1-|\xi|^{2s}+C_{r_0}|\xi|^{2s+\delta}}\le \abs{1-\frac{1}{2}|\xi|^{2s}}\le e^{-|\xi|^{2s}}\quad\text{ for }|\xi|\le r
$$
since for all $x\in (0,1)$, $|1-x|\le e^{-x}$
\item This statement is a direct consequence of \citep[Prop. 26]{carlen_entropy_2010}, where the assumption of
  finiteness of the entropy comes from the assumption that 
  $f \in L^1\cap L^p$ and has some finite moment.
  \qedhere
\end{enumerate}

\end{proof}

\begin{proof}
    Again we follow the same procedure of \citep{canizotassi2024}.
    \begin{equation*}
    	\begin{split}
    		{\Jnormh}(\xi)&=\int e^{-ix\xi} (\bar \sigma^{2s} n)^{d/{2s}}f^{\sigma_1\dots,\sigma_n}(n^{1/2s}\bar \sigma x)\d x=\int e^{-iy \frac{\xi}{{n}^{1/2s}\bar \sigma}}f_{\sigma_1}*\cdots*f_{\sigma_n}(y)\d y\\
    		&=\mathcal{F}(f_{\sigma_1}*\cdots f_{\sigma_n})\Big(\frac{\xi}{n^{1/2s}\bar \sigma}\Big)=\prod_{j=1}^n\hat{f}_{\sigma_i}\Big(\frac{\xi}{n^{1/2s}\bar \sigma}\Big)=\prod_{j=1}^n \int e^{-ix \frac{\xi}{n^{1/2s}\bar \sigma}}\sigma_i^{-d}f\Big(\frac{x}{\sigma_i}\Big) \d x \\
    		&=\prod_{j=1}^n \int e^{-iy\frac{\xi \sigma_i}{n^{1/2s}\bar \sigma}}f(y)\d x=\prod_{j=1}^n\JJi.
    	\end{split}
    \end{equation*}
Since $\widehat{G^s}(\xi)=e^{-|\xi|^{2s}}$ also
\begin{equation*}
 \widehat{G}_s(\xi)=\prod_{i=1}^n \GGi.
\end{equation*}

Without loss of generality, we can assume $p\le2$. Hausdorff--Young's
inequality then implies that $\hat{f}$ is in $L^{p'}\cap L^\infty$
with $p'\in [2,\infty)$. 
 Using H\"older's inequality,
$$\norm{\prod_{j=1}^n\JJi}_{L^1}\le \prod_{j=1}^n\norm{\JJi}_{L^n}=
\prod_{j=1}^n\Big(\frac{\bar\sigma
  \sqrt{n}}{\sigma_j}\Big)^{d/n}\norm{\hat{f}}_{L^n}\le
n^{d/2}\Big(\frac{L}{l}\Big)^d\norm{\hat{f}}_{L^n}^n,$$
which is
finite for $n\ge p'$, i.e.  $\Jnormh$ is in $L^1$ for any $n\ge
p'$, i.e we control the $L^\infty$ norm by the $L^1$ norm of the
Fourier transform, since
\begin{equation*}
  \abs{\Jnorm(x)-G(x)}
  =
  (2\pi)^{-d}\abs{\ird (\Jnormh (\xi)-\hat G(\xi))e^{i\xi x}\d \xi}
  \le
  (2\pi)^{-d}\ird \abs{\Jnormh(\xi)-\hat G(\xi)} \d \xi
\end{equation*}

Let $r $  be given by Lemma \ref{lemfou} $(i)$ for $\hat f$ and define 
$B_r=B_{r}(n,l,L):=\{ \xi: \abs{\xi}<n^{1/2s} \frac{l}{L}r \}$
and  we split the integral into high and low frequencies (we drop the $(2\pi)^{-d}$ constant for better readability)
\begin{equation*}
    \norm{\Jnorm-G}_{L^\infty}\le \int_{\xi\in B_r^c}|\Jnormh(\xi)|\d \xi+\int_{\xi\in B_r^c}|\hat G(\xi)|\d \xi+\int_{\xi\in B_r}|\Jnormh(\xi)-\hat G(\xi)|\d \xi:=T_1+T_2+T_3
\end{equation*}
We suppose without loss of generality $p'$ integer. If this is not the case, we work with $\cl{p'}$ since $\hat{f}$ is also in  $L^{\cl{p'}}$.
\begin{equation}\label{dod}
    \begin{split}
        &\int_{B_ r^c} \prod_{j=1}^n\abs{\JJi}= n^{d/2s}\int_{|\eta| \ge\frac{l}{L} r} \abs{\hat f \Big(\frac{\sigma_1}{\bar \sigma} \eta\Big)}\cdots 
        \abs{\hat f\Big(\frac{\sigma_{p'}}{\bar \sigma} \eta\Big)} 
        \abs{\hat f \Big(\frac{\sigma_{p'+1}}{\bar \sigma} \eta\Big)}\cdots \abs{\hat f\Big(\frac{\sigma_n}{\bar \sigma} \eta\Big)} \d\eta\\
        &\le n^{d/2s}  \sup_{|\eta| \ge\frac{l}{L} r} \abs{\hat f \Big(\frac{\sigma_{p'+1}}{\bar \sigma} \eta\Big)}\cdots \sup_{|\eta| \ge\frac{l}{L} r} \abs{\hat f\Big(\frac{\sigma_n}{\bar \sigma} \eta\Big)} \int_{|\eta| \ge\frac{l}{L} r} \abs{\hat f \Big(\frac{\sigma_1}{\bar \sigma} \eta\Big)}\cdots 
        \abs{\hat f\Big(\frac{\sigma_{p'}}{\bar \sigma} \eta\Big)}  \d\eta
    \end{split}
\end{equation}
For each $j=p'+1,\dots,n$
\begin{equation}\label{1313}
    \sup_{|\eta| \ge\frac{l}{L} r} \abs{\hat f \Big(\frac{\sigma_{j}}{\bar \sigma} \eta\Big)}\le 
    \sup_{ |\eta'| \ge r\frac{l^2}{L^2}} |\hat f (\eta')|\le \kappa
\end{equation}
with $\kappa=\kappa(r\frac{l^2}{L^2})\in (0,1)$, coming from Lemma \ref{lemfou} $(i)$.

The  factors in the integrals in equation \eqref{dod} can be bounded again via H\"older's Inequality and thus we rewrite 
\begin{equation}\label{1312}
\begin{split}
    &\int_{|\eta| \ge\frac{l}{L}r} \prod_{j=1}^{p'}\abs{\hat f \Big(\frac{\sigma_j}{\bar \sigma} \eta\Big)}  \d\eta\le \prod_{j=1}^{p'}\bigg(\int_{|\eta| \ge\frac{l}{L}r}\abs{\hat f \Big(\frac{\sigma_j}{\bar \sigma} \eta \Big)\d\eta}^{p'}\bigg)^{1/p'}\le \Big(\frac{L}{l}\Big)^{d}||\hat{f}||_{p'}^{p'}\le\Big(\frac{L}{l}\Big)^{d}C_p||f||_{p}^{p'} 
        \end{split}
\end{equation}
 applying first H\"older  and then Hausdorff-Young inequality ($C_p$ is the Hausdorff-Young constant). Substituting equations \eqref{1313} and \eqref{1312} into equation \eqref{dod}, we arrive at
\begin{equation}\label{eq: 21}
    T_1\le C_pn^{d/2s}\Big( \frac{L}{l}\Big)^{d} \kappa^{n-p'}||f||_p^{p'}
\end{equation}
	Applying the same procedure, we  derive a very similar bound for  $T_2$ as well.
	Since $\kappa<1$, the term $\kappa(r)^{n-p'}$ decays rapidly with respect to 
 $n$, thus for $n$ large enough (depending on $p$), there exists $C_1>0$ such that
	\begin{equation}\label{eq: 21bis}
		T_1+T_2\le \frac{C_1}{n^{\delta/{2s}}}
	\end{equation}

We are left to analyse $T_3$.
 Using the factorisation in Lemma \ref{lemma: produttoria}, we obtain
 
\begin{equation}\label{eqA}
    \begin{split}
&\abs{\Jnormh(\xi)-\hat G^s}=\prod_{j=1}^n \JJi-\prod_{j=1}^n \GGi\\
&=\sum_{j=1}^n \abs{\JJi-\GGi}\abs{\prod_{k=1}^{j-1}\JJj}\abs{\prod_{k=j+1}^{n}\GGj}\\
    \end{split}
\end{equation}

Since $$\xi\in B_r\implies \abs{\frac{\xi\sigma_k}{n^{1/2s}\bar{\sigma}}}\le r\frac{\sigma_k}{\bar{\sigma}}\frac{l}{L}\le r$$
from Lemma \ref{lemfou} $(i)$ it follows that 
\begin{equation*}
    \abs{\hat{f}\Big(\frac{\xi \sigma_j}{n^{1/2s}\bar\sigma }\Big)}\le e^{-\frac{|\xi|^{2s}}{2n}\frac{\sigma_i^{2s}}{\bar \sigma^{2s}}}\le e^{-\frac{|\xi|^{2s}}{2n}(\frac{l}{L})^{2s}}. 
\end{equation*}
and 
\begin{equation*}
    \abs {\widehat{G}^s\Big({\frac{\xi \sigma_j}{n^{1/2s}\bar \sigma}}\Big)}=e^{-\frac{|\xi|^{2s}}{n}\frac{\sigma_j^2}{\bar \sigma^2}}\le e^{-\frac{|\xi|^{2s}}{n}(\frac{l}{L})^{2s}  }.
\end{equation*}

Then
\begin{equation*}
\begin{split}
    &\frac{\abs{\Jnormh(\xi)-\hat G^s(\xi)}}{|\xi|^{2s+\delta}}=\frac{1}{\abs{\xi}^{2s+\delta}}\Big[ \sum_{j=1}^n \abs{\JJi-\GGi}\abs{\prod_{k=1}^{j-1}\JJj}\abs{\prod_{k=j+1}^{n}\GGj}\Big]\\
    &\le \frac{1}{\abs{\xi}^{2s+\delta}} \sum_{j=1}^n \abs{\JJi-\GGi} e^{-\frac{|\xi|^{2s}}{2n}(j-1)(\frac{l}{L})^{2s}}e^{-\frac{|\xi|^{2s}}{n}(n-j)(\frac{l}{L})^{2s}  }\\
    &\le \frac{1}{\abs{\xi}^{2s+\delta}} \sum_{j=1}^n \abs{\frac{\xi \sigma_j}{n^{1/2s}\bar \sigma}}^{2s+\delta} \frac{\abs{\JJi-\GGi}}{ \abs{\frac{\xi \sigma_j}{{n}^{1/2s}\bar \sigma}}^{2s+\delta}}e^{-\frac{|\xi|^{2s}}{4}\frac{n-1}{n}(\frac{l}{L})^{2s}}\\
    &\le \frac {1}{n^{1+\frac{\delta}{2s}}}\Big(\frac{L}{l}\Big)^{2s+\delta}n \sup_{|\eta'|\le r}\frac{|\hat{f}(\eta')-\hat G^s(\eta')|}{\abs{\eta'}^{2s+\delta}} e^{-\frac{\abs{\xi}^{2s}}{4}(\frac{l}{L})^{2s} }
    \end{split}
\end{equation*}
From Lemma  \ref{lemfou}, since $|\eta'|\le r\le r_0$
$$
\sup_{|\eta'|\le r }\frac{|\hat{f}(\eta')-\hat G^s(\eta')|}{\abs{\eta'}^{2s+\delta}} \le C_{r_0}
$$
Summing up the results, the decay rate of the term $T_3$ is
\begin{equation*}
    \begin{split}
        T_3&\le \frac{C_{r_0}}{n^{\delta/2s}}\Big(\frac{L}{l}\Big)^{2s+\delta} \int e^{-\frac{\abs{\xi}^{2s}}{4}(\frac{l}{L})^{2s} }|\xi|^{2s+\delta}\d \xi:=\frac{\tilde{C}}{n^{\delta/2s}}.
    \end{split}
\end{equation*}
Including the estimates for $T_1, T_2$, we complete the proof.

\end{proof}

\begin{osservazione}\label{rmk: CLTunif}
Standard Generalised Central Limit Theorems, such as those established in \citep{goudon_junca_toscani_BE_2002}, typically rely on assumptions similar to Hypothesis \ref{assJ}, proving a convergence rate that depends explicitly on the  gap $1-s$. 
In contrast, we emphasise again that thanks to Hypothesis \ref{ass: sunif} we establish a Generalised Central Limit Theorem  with a uniform--in--$s$ rate of convergence $\delta$. This stability is obtained by restricting the class of functions, which in turn allows us to obtain a speed of convergence that remains positive in the limit $s\to 1.$ Although based on a simple premise, to the best of our knowledge, no similar uniform result  exists in the literature.
\end{osservazione}


\section{Asymptotic behaviour and properties of the equilibrium}\label{sec: asymptotic}
 To prove Theorem \ref{thm: main1}, we  use Harris's theorem to show exponential convergence to equilibrium of solutions in weighted $L^1_k$ spaces independently of the scaling parameter $\e\in(0,1)]$ and of the fractional index $s\in(1/2,1)$.
\subsection{Harris's Theorem }\label{harristheory}
For completeness, we briefly summarise  the main concepts of Harris’s theorem below, referring the reader to \citep{Hairer2011, canizo_harris-type_2021, Meyn2010} for a more detailed treatment.

Let $\M = \M(\Omega)$ be the space of finite signed measures on a measurable space $\Omega$, with total variation norm $\norm{\mu} = \int |\mu|$. Denote by $\mathcal{P}$ the set of probability measures in $\M$. Given a measurable function $V:\Omega\rightarrow[1,\infty)$ 
(a weight function), define $\M_V$ as the subspace of measures for which $\norm{\mu}_V := \int V |\mu| < \infty$, and set $\mathcal{P}_V := \M_V \cap \mathcal{P}$.

A linear operator $S : \M \to \M$ is called stochastic if it preserves probability measures. A stochastic semigroup $(S_t){t \ge 0}$ is a family of such operators satisfying $S_0 = \mathrm{Id}$ and $S_{t+s} = S_t \circ S_s$ for all $t, s \ge 0$.

We say that $S$ is stochastic on $\M_V$ if it maps $\M_V$ into itself and is bounded in the $\norm{\cdot}_V$ norm. The same applies to a stochastic semigroup.

A common tool to establish exponential convergence to equilibrium for Markov semigroups is Harris’s Theorem, which combines a confining (Lyapunov-type) and a minorisation condition (local positivity), called Doeblin-Harris or Harris Condition.

\begin{hyp}[Semigroup Lyapunov Condition]
There exist constants $\lambda, C > 0$ such that
\begin{equation}\label{eq: semigroupLyap}\tag{LC}
\norm{S_t \mu}_V \le e^{-\lambda t} \norm{\mu}_V + \frac{C}{\lambda} (1 - e^{-\lambda t}) \norm{\mu}.
\end{equation}
\end{hyp}

\begin{hyp}[Harris Condition]
For some set $\mathcal{S} \subseteq \Omega$, $S$ satisfies
\begin{equation}\label{eq: HLB}\tag{HC}
	S\mu \ge \alpha \rho \int_{\mathcal{S}} \mu \quad \text{for all } 0 \le \mu \in \M,
\end{equation}
for some $\alpha \in (0,1)$ and a probability measure $\rho$.
\end{hyp}

\begin{teorema}[Harris's theorem for semigroups]
\label{sHarris}
Let $V : \Omega \to [1, \infty)$ be a measurable weight function, and let $(S_t)_{t \ge 0}$ be a stochastic semigroup on $\M_V$. Suppose that the following hold:
\begin{enumerate}
	\item The semigroup $(S_t)_{t \ge 0}$ satisfies the semigroup Lyapunov Hypothesis \ref{eq: semigroupLyap}.
	\item There exist constants $T^* > 0$ and $R > \frac{2C}{\lambda}$ such that the operator $S_{T^*}$ satisfies the Harris condition Hypothesis \ref{eq: HLB} on the set $\mathcal{S} := \{ x \in \Omega : V(x) \le R \}$.
\end{enumerate}
Then the semigroup admits a unique invariant probability measure $\mu^* \in \mathcal{P}_V$, and there exist constants $C, \lambda > 0$ such that for any $\nu \in \mathcal{P}_V$ with $\int_\Omega \nu = 0$, and for all $t \ge 0$, we have
\begin{equation*}
	\norm{S_t \nu}_V \le C e^{-\lambda t} \norm{\nu}_V.
\end{equation*}
In particular, for any $\mu \in \mathcal{P}_V$ and all $t \ge 0$, choosing $\nu=\mu-\mu^*$, one has
\begin{equation*}
	\norm{S_t \mu - \mu^*}_V \le C e^{-\lambda t} \norm{\mu - \mu^*}_V.
\end{equation*}
\end{teorema}

If $S_t$ is generated by an operator $L$, a more familiar form of the Lyapunov condition is
\begin{equation}
	\label{eq:sLyap}
	L^*V \le C - \lambda V,
\end{equation}
where $L^*$ is the (formal) adjoint acting on suitable test functions (see \citep{hairer2010convergence}, where this is interpreted in a submartingale sense). We verify this condition for the generator of equation \eqref{DFFP} in Theorems \ref{thm: lyapunov}, which in turn implies Hypothesis \ref{eq: semigroupLyap}.

\subsection{Lyapunov Condition}\label{subsec: lyap}
We define  $L_\e^*$, the formal dual operator of $L_e^s$ acting on $\mathcal{C}^2$ functions with appropriate decay at infinity to ensure that the dual operator of the convolution is well-defined. To lighten the notation we will remove the superscript $s$.
\begin{equation}\label{eq: dualL}	    L_\e^*\f=\frac{1}{\e^{2s}}\ird J_\e^s(y)[\f(x+y)-\f(x)]\d y-x\nabla \f.
     \end{equation}

\begin{teorema}\label{thm: lyapunov}
  Let $0<k<2s$ and assume that $J$ satisfies \ref{ass: sunif}.
  Then, for all $s\in(0,1)$ $\e\in(0,1]$, there exist constants $C_L, \lambda_L>0$ that depend only on $k$, $d$ and $J$ but not on $\e$ or $s$
  such that
  $$L_\e^*(\ap{x}^k)\le C_L-\lambda_L \ap{x}^k.$$
\end{teorema}

\begin{proof}

\begin{itemize}
\item \emph{Step 1: Fractional diffusion operator}

We start by analyzing the dual  of the fractional diffusion operator $\mathcal{A}_\e^*$ by splitting it into near and far field. After a change of variable we have
\begin{multline*}
\mathcal{A}_\e^*\f
=\frac{1}{\e^{2s}}\ird J^s(y)(\f(x+ \e y)-\f(x))\d y\\
=\frac{1}{\e^{2s}}\int_{|y|\le R} J^s(y)(\f(x+ \e y)-\f(x))\d y
+\frac{1}{\e^{2s}}\int_{|y|>R} J^s(y)(\f(x+ \e y)-\f(x))\d y,
\end{multline*}
where $R$ is the one appearing in \eqref{eq: tailpointwise}.
\begin{itemize}
    \item \emph{Step 1a: Near field estimate.}
         Since $\f$ is (at least) $C^2$,  Taylor's theorem gives 
    \begin{equation}\label{eq: taylor}
     \f(x+\e y)-\f(x)=\e \nabla \f(x)\cdot y+\frac{\e^2}{2}y^TD^2 \f(\zeta)y
    \end{equation}
    for some $\zeta\in[x,x+\e y]$.
     But
    $$
    D^2\f(\zeta)= k\Id \ap{\zeta}^{k-2}+k(k-2)z\otimes \zeta\ap{\zeta}^{k-4}.
    $$
    Hence, taking the absolute value, factoring out the common terms, and taking into account that $0<k<2$
    \begin{multline*}
    	\abs{D^2\f(\zeta )}\le k(1+|\zeta |^2)^{\frac{k}{2}-1}+|k(k-2)|(1+|\zeta |^2)^{\frac{k}{2}-2}|\zeta |^2
    	\\
    	= k(1+|\zeta |^2)^{\frac{k}{2}-1}+|k(2-k)|(1+|\zeta |^2)^{\frac{k}{2}-1}\frac{|\zeta |^2}{1+|\zeta |^2}\\
        \le \big(k+k(2-k)\big)(1+|\zeta |^2)^{\frac{k}{2}-1}\le C_k(1+|\zeta |^2)^{\frac{k}{2}-1}\le C_k
    \end{multline*}   

Thus since $J^s$ is centered,
\begin{multline}
\frac{1}{\e^{2s}}\int_{|y|\le R}J^s(y)(\f(x+\e y)-\f(x))\d y \\
=\frac{1}{\e^{2s}}\int_{|y|\le R}J^s(y) \nabla \f(x)\cdot \e y \d y+\frac{1}{\e^{2s}}\int_{|y|\le R}J^s(y)\frac{\e^2}{2}y^TD^2\f[\zeta]y\f(y)\d y\\
\le 
\e^{2-2s}\frac{1}{2}\int_{|y|\le R} J^s(y)\norm{D^2\f}_{L^\infty(B_R(x))}| y|^2\d y\\
    \le C_k\e^{2-2s}\int_{|y|\le R}|y|^{2}J^s(y)\d y =C_k\e^{2-2s}M_2^R(J^s)
\end{multline}
where $M_2^R(J^s)=\int_{|y|\le R}|y|^2 J^s(y)\d y<\infty$ by \eqref{eq: levydensity}.
\item {\textit{Step 1b: Far field estimate.}}
Since $|y|>R$, by Hypothesis \ref{ass: sunif}, \eqref{eq: tailpointwise}
\begin{multline}
    \frac{1}{\e^{2s}}\int_{|y|> R}J^s(y)(\f(x+\e y)-\f(x))\d y \\
    \le \frac{C}{\e^{2s}}(1-s)\ird |y|^{-d-2s} (\f(x+\e y)-\f(x))\d y + \frac{1}{\e^{2s}}\ird \Psi(y)(\f(x+\e y)-\f(x))\d y 
\end{multline}
By change of variable one can easily see that 
\begin{multline*}    
\frac{C}{\e^{2s}}(1-s)\ird |y|^{-d-2s} (\f(x+\e y)-\f(x))\d y\\
=\frac{C}{\e^{2s}}(1-s)\ird \e^{-d}\frac{\f(x+z)-\f(x)}{|z/\e|^{d+2s}}\d z\lesssim \ap{x}^{k-2s}
\end{multline*}
from the properties of the fractional Laplacian (see Proposition \ref{prop: fraclap_decay}.
Note that the scaling of the fractional operator exactly compensates for the $\e^{-2s}$ factor.

The second term can be bounded with no difficulties in $\e$ since $\Psi$ has finite second moment.
\end{itemize}
\item \emph{Step 2: Drift term.}
\begin{equation*}
\mathcal{B}^*\f=-x\cdot\nabla\f(x)=-k|x|^2(1+|x|^2)^{k/2-1}=-k\ap{x}^k+k\ap{x}^{k-2}
\end{equation*}
\item \emph{Step 3: Lyapunov.} Combining all these estimates,
$$L_\e^s\f=\mathcal{A}_\e^*\f+\mathcal{B}^*\f\le C_k\e^{2-2s}M_2^R(J^s)+ C\ap{x}^{k-2s}+\e^{2-2s}M_2(\Psi)-k\ap{x}^k+ k\ap{x}^{k-2}\le C_L-\lambda_L\ap{x}^k
$$
The constants $C$ and $\lambda$ can be clearly taken uniform in $\e\in(0,1]$ and $s\in(1/2,1]$.
    \end{itemize}
\end{proof}

\begin{proposizione}\label{prop: fraclap_decay}
Let $s\in(0,1)$. Then
    $$
    -(-\Delta)^s\ap{x}^{k}\lesssim \ap{x}^{k-2s}\left(1+\frac{1}{2s}+\frac{1}{2s-k}\right)
    $$
\end{proposizione}
\begin{proof}[Sketch of the proof]
Split the integral into two parts: $\ap{y}\le \ap{x}$ and $\ap{y}> \ap{x}$.
For the first one, proceed with a Taylor expansion bound and notice that $\ap{\zeta}^{k-2}\lesssim \ap{x}^{k-2}$. For the second one
$$
\int_{\ap{y}\ge \ap{x}}\frac{\ap{x+y}^k}{|y|^{d+2s}}\lesssim \ap{x}^k\int_{\ap{y}\ge \ap{x}}|y|^{-(d+2s)}\d y +\int_{\ap{y}\ge \ap{x}}\ap{y}^{-(d+2s-k)}\lesssim \frac{1}{2s}\ap{x}^{k-2s}+\frac{1}{2s-k}\ap{x}^{k-2s}
$$
    
\end{proof}

\begin{teorema}[Semigroup Lyapunov condition.] \label{thm: semigrouplyap}
	Assume $u_0$ and $J^s$
	satisfy the hypotheses of Theorem \ref{thm: lyapunov}. For $k<2s$ and  $\f=\ap{x}^k$, the semigroup associated to
	equation \eqref{DFFP} satisfies a Lyapunov condition 
	\eqref{eq: semigroupLyap} with $V\equiv \f$ and the same $C_L$ and $\lambda_L$ of Theorem
	\ref{thm: lyapunov}.
\end{teorema}
\begin{proof}
	Let  $u_0$ be in $\mathcal{D}(L^s_\e)$. By standard semigroup theory
	$ u(t)\in \mathcal{C}^1([0,\infty);  L^1_k)$. Hence, for regular $\f$ we may differentiate under the integral sign and obtain 
	\begin{equation*}
			\ddt \ird u(t,\cdot) \f(x)\d x= \ird \ddt u(t,\cdot) \f(x) \d x  \\
			=\ird (L_\e^s u) \f (x)
			\end{equation*}
			Integrating by parts (justified by standard cutoff/bump function argument)
			gives 
			$$
			\ird (L_\e^s u) \f (x)=\int u(L_\e^*\f).
			$$
	
Applying Gronwall Lemma with the $\mathcal{C}^1$ function $f(t):= \ird u(t)\f$:
	that is,$$
	f'(t)\le -\lambda_L f(t)+ C_L\implies f(t)\le e^{-\lambda t}f(0)+ \frac{C_L}{\lambda_L}(1-e^{-\lambda t}).$$
	
	Since the $\mathcal{D}(L_\e^s)$ is dense in $L^1_k$, we extend the result to  any $u_0\in L^1_k$ by standard density argument.
\end{proof}

\subsection{Harris's Condition}\label{subsec: positivity}
Harris condition \ref{eq: HLB} is a uniform  local positivity condition at a fixed positive time. We are able to show a stronger result,i.e. a local positivity property at any positive time $t>0$.
\begin{teorema}[Lower bound]\label{posit}
		Let $(J^s)$ be a family of kernels satisfying Hypothesis \ref{ass: sunif}. Then, the solution $u$ of equation \eqref{DFFP} with a probability initial condition $u_0$ is uniformly positive on compact sets for all $\e\in(0,1],$ and $s\in(1/2,1)$. That is, for all $t \ge 0$, and $R_1, R_2 > 0$, there exists a constant $\alpha > 0$, independent of $\e$ and $s$, such that the solution of \eqref{DFFP} with initial data $u_0 \in L^1_k$ satisfies
		$$
		u(t,x) \ge \alpha \int_{B_{R_2}} u_0 \qquad \text{for all } x \in B_{R_1}.
		$$	
\end{teorema}	

First, let us define the following function, which will be used to make the lower bound independent of $s$ when needed.
\begin{definizione}\label{def: lbStables} We define $H(x)$ as (any) function that is pointwise below $G^s$ for all $s\in(1/2,1)$. A concrete example is obviously $$
    H(x)=\min_{s\in[1/2,1]}G^s(x)
    $$
    which is strictly positive since $G^S$ is continuous in  $s$ and strictly positive on $\R^d$.
    Notice that we also included in the minimum, the cases $s=1/2$ and $s=1$ (resp. Cauchy and Gaussian distribution)
    A more explicit function can be for example $$
    H(x)=(4\pi)^{-d/2}\min\{e^{-R^2/4}, e^{-|x|^2/4}\}
    $$
    for $R$ large enough.
\end{definizione}

The proof of  Theorem \ref{posit} relies on the Wild sums representation and the $L^\infty$ version of the generalised Berry--Esseen theorem introduced earlier. As in our previous work \citep{canizotassi2024}, we distinguish between two regimes: one where $\e$ lies in a compact interval bounded away from zero, and another where $\e$ is small. These two cases are handled separately in the following two lemmas.

\begin{lemma}\label{lemma: small eps}
    Let $(J^s)$ be a family of kernels satisfying Hypothesis \ref{ass: sunif}. Then for all $t, \eta > 0$, there exists $\e_0 > 0$ such that for all $\e \in (0, \e_0)$, there exists a constant $A > 0$, independent of $\e$ and $s$, such that the following holds: for all integers $n$ satisfying
$$
\frac{t}{\e^{2s}} \le n \le 2\frac{t}{\e^{2s}},
$$
and for all sequences $0 \le t_1 \le t_2 \le \dots \le t_n \le t$, we have
$$
\Jn(x) \ge A >0 \qquad \text{for all } x \in B_\eta.
$$
\end{lemma}
 \begin{proof}
Let $N$ be the threshold from Theorem \ref{thm: BE}. Define $\e_1 := (t/N)^{1/2s}$, ensuring $n \ge t/\e^{2s} \ge N$ for all $\e < \e_1$.

Define $$
\bar{\tau}^{2s}=\frac{1}{n}\sum_{i=1}^n (e^{t_i})^{2s}
$$

Since $C_{BE}$ and $\delta$ are $s$-independent (by Hypothesis \ref{ass: sunif}), and $t$ is fixed, we can choose $\e_2$ small enough such that:
$$\frac{C_{BE}}{t^{\delta/2s}}\e_2^{\delta} \le \frac{1}{2}H\left(\frac{\eta}{t^{1/2s}\bar{\tau}}\right).$$

Take $\e_0 := \min\{\e_1, \e_2\}$. 

Then for $\e < \e_0$, and $x \in B_\eta$, we use the GCLT estimate:
\begin{align*}
    \Jn(x) &\ge (\e^{2s}n\bar{\tau}^{2s})^{-d/2s} \left[ G^s\left(\frac{x}{n^{1/2s}\e\bar{\tau}}\right) - \frac{C_{BE}}{n^{\delta/2s}} \right] \\
    &\ge (2te^{2st})^{-d/2s} \left[ G^s\left(\frac{\eta}{t^{1/2s}\bar{\tau}}\right) - \frac{C_{BE}}{t^{\delta/2s}}\e^\delta \right] \\
    &\ge (2te^{2t})^{-d/2s} \frac{1}{2} H\left(\frac{\eta}{t^{1/2s}\bar{\tau}}\right).
\end{align*}

Finally, note that $(2te^{2t})^{-d/2s} \ge \min\{ (2te^{2t})^{-d}, (2te^{2t})^{-d/2} \}$, 

Since $H$ is a fixed, radially decreasing function, and $t^{1/2s}$ is bounded between $\sqrt{t}$ and $t$, we have $H(\frac{\eta}{t^{1/2s}\bar{\tau}}) \ge \min\{ H(\frac{\eta}{t\bar{\tau}}), H(\frac{\eta}{\sqrt{t}\bar{\tau}}) \}$ (depending whether $t\ge 1$).
Thus

$$
\Jn(x)\ge \frac{1}{2}\min\{ (2te^{2t})^{-d}, (2te^{2t})^{-d/2} \}\min\{ H(\frac{\eta}{t\bar{\tau}}), H(\frac{\eta}{\sqrt{t}\bar{\tau}}) \}:=A
$$

This provides a constant $A$ depending only on $t, \eta, d, C_{BE}$ and the bounds of $H$, but independent of $s$ and $\e$.

\end{proof}

\begin{lemma}\label{lemma: large eps}
     Let $(J^s)$ be a family of kernels satisfying Hypothesis \ref{ass: sunif}, and assume $\e\in[\e_0,1]$, for some $\e_0$. Then for all $t>0,\eta>0$ there exists a constant $B>0$ ---independent of $s$ and $\e$ ---such that 
    $$
    \Jn(x)>\frac{B}{n^{d/2s}}\qquad\text{for all }x\in B_{\eta}\text{ and }0\le t_i\le \dots\le t_n\le t
    $$
\end{lemma}

\begin{proof}
	Let $N=\max\{N_1,N_2\}$,
	where $N_1$ is the threshold for which estimate in Theorem \ref{thm: BE} holds, and $N_2$ is such that  $$G^s\Big(\frac{\eta}{\e_0}\Big)-\frac{C_{BE}}{N_2^{\delta/2s}}\ge\frac{1}{2}G^s\Big(\frac{\eta}{\e_0}\Big).$$
	Then for $n\ge N$, applying Berry -- Esseen CLT, we obtain
	
	\begin{multline*}
			\Jn (x)
			\ge (\e^{2s}
			n\bar\tau^{2s})^{-d/2s}\Big[G^s\Big(\frac{x}{n^{1/2s}\e\bar\tau}\Big)-\frac{C_{BE}}{n^{\delta/2s}}\Big]
			\\
			 \ge \frac{1}{n^{\frac{d}{2s}}}\Big(\frac{1}{2^{1/2s}e^t}\Big)^d\Big[G^s\Big(\frac{\eta}{\e_0}\Big)-\frac{C_{BE}}{N_2^{\delta/2s}}\Big]\\
			 \ge \frac{1}{n^{\frac{d}{2s}}}\Big(\frac{1}{2^{1/2s}e^t}\Big)^d\frac{1}{2}G^s\Big(\frac{\eta}{\e_0}\Big)=:\frac{B}{n^{\frac{d}{2s}}}.
	\end{multline*}   
\end{proof}

We recall the following trivial lemma on Poisson distribution, already proved in \citep{canizotassi2024}
\begin{lemma}
  \label{2n}
  There exists an explicit constant $C_L$ such that
  $$s_m := e^{-m}\sum_{n=m}^{2m}\frac{m^n}{n!}\ge C_L
  \qquad \text{for all $m \geq 1$}$$ 
\end{lemma}

We now have all the ingredients to prove Theorem \ref{posit}.

\begin{proof}
	[Proof of Theorem \ref{posit}]	
	Let $u(t,x)$ be the solution of \eqref{DFFP} with probability initial data $u_0$. By the Wild sums representation and since $u_0\ge0$ and $J^s_\e\ge0$, it follows that
	$$
	u(t,x)\ge e^{(d-\frac{1}{\e^{2s}})t}\sum_{n=1}^\infty \Big(\frac{1}{\e^{2s}}\Big)
	\int_0^t\int_0^{t_n}\cdots\int_0^{t_2}
	\int_{\R^d}
	J_\e^{t_1,\dots,t_n}\big(e^t x - y\big)\,u_0(y)\,\d y
	 \d t_1\cdots \d t_n.
	$$
	Set
	$$
	\eta  :=  e^t R_1+R_2,
	$$
	such that for  $x\in B_{R_1}$ we have $e^t x-y\in B_\eta$.
	
	\medskip
	
	\noindent\textbf{Case 1: small $\mathbf{\e}$.}
	For $t$, $\eta$, let $\e_0$ be as in Lemma \ref{lemma: small eps}, and define
	$$
	\tilde\e_0  :=  \min\{\e_0,\sqrt{t}\}.
	$$
	If $0 < \e \le \tilde\e_0$, then Lemma \ref{lemma: small eps} gives a constant $A>0$ such that for every integer 
	$$
	m  =  \Bigl\lfloor \tfrac{t}{\e^{2s}}\Bigr\rfloor,
	\qquad
	m \le n \le 2m,
	\qquad
	y\in B_{R_2},
	$$
	we have
	$$
	J_\e^{t_1,\dots,t_n}(e^t x-y)
	\ge
	A
	$$
	Hence, 	using the estimate
	from Lemma \ref{2n}
	\begin{multline*}
		u(t,x) \ge	A\,e^{dt}e^{-\frac{t}{\e^{2s}}} \sum_{n=m}^{2m}
		\Big(\frac{t}{\e^{2s}}\Big)^n\frac{1}{n!}
		\int_{B_{R_2}}u_0(y) \d y
		\\
		\ge A\,e^{dt-1} e^{-m} \sum_{n=m}^{2m} \frac{m^n}{n!} \int_{B_{R_2}}u_0(y) \d y \ge
	A\,C_L\,e^{dt-1}
	\int_{B_{R_2}}u_0(y)\,dy.
\end{multline*}
	Set
	$$
	\alpha_1  :=  A\,C_L\,e^{dt-1}.
	$$
	
	\medskip
	
	\noindent\textbf{ Case 2: large $\e$.}
	If $\e \ge \tilde\e_0$, then in particular $\e\ge \e_0$ and we may apply Lemma \ref{lemma: large eps} that gives $N\in \N$  and $B>0$ such that for all $n\ge N$ and $y\in B_{R_2}$,
	$$
	\int_0^t\int_0^{t_n}\cdots\int_0^{t_2}
	J_\e^{t_1,\dots,t_n}(e^t x - y)\,\d t_1\cdots \d t_n
	 \ge  B
	\frac{t^n}{n!}.
	$$
	Therefore
	\begin{equation*}
		u(t,x)
		\ge
		B e^{(d-\frac1{\e_0^{2s}})t}
		\sum_{n=N}^\infty
		\frac{t^n}{n!}\,\frac{1}{n^{d/2s}}
		\int_{B_{R_2}}u_0(y)\,\d y
		\\
		\ge	B\,e^{(d-\frac1{\tilde\e_0^{2s}})t}
		\frac{t^N}{N!}\,\frac{1}{N^{d/2s}}
		\int_{B_{R_2}}u_0(y)\,\d y.
	\end{equation*}
	Set
	$$
	\alpha_2=B\,e^{(d-\frac1{\tilde\e_0^{2s}})t}
	\frac{t^N}{N!}\,\frac{1}{N^{d/2s}}
	$$
	
	\medskip
	
	Finally, let
	$$
	\alpha  := \min\{\alpha_1,\alpha_2\}.
	$$
	Then for every $0<\e\le1$ and all $x\in B_{R_1}$,
	$$
	u(t,x) \ge \alpha\int_{B_{R_2}}u_0(y)\,\d y,
	$$
	
\end{proof}

\paragraph*{Proof of Theorem \ref{thm: main1}}
\begin{proof}The proof follows directly from the application of Harris's Theorem  explained in Section \ref{harristheory}. Theorem \ref{thm: semigrouplyap} is the Lyapunov condition \eqref{eq: semigroupLyap} 
with $V(x) = \ap{x}^k$ and constants $C_L, \lambda_L > 0$ that are independent of $\e$ and $s$.
Next, we choose $R_1 > 0$ large enough so that the set $\mathcal{C} = \{x : V(x) \le R\}$ required by Theorem \ref{sHarris} is contained in $B_{R_1}$.
	From Theorem \ref{posit}, for all $t>0$, and all radii $R_1,R_2>0$, the operator $S_t$ satisfies Harris condition \eqref{eq: HLB}: for $u_0\in L^1_k$
	$$
	S_tu_0\ge \alpha \mathds{1}_{B_{R_1}}\int_{B_{R_2}}u_0(y)\d y
	$$
    with a constant which is independent of $\e$ and $s$. 
    Then, Theorem \ref{sHarris} ensures the existence and uniqueness of the equilibrium in $L^1_k$ and that $$\norm{u(t,\cdot)-F^s_\e}_{L^1_k} \le C e^{-\lambda t} \norm{u_0 - F^s_\e}_{L^1_k}, \quad \forall t \ge 0$$
    with $C$ and $\lambda$ independent of $\e$ and $s.$
\end{proof}


\subsection{Properties of the equilibrium}\label{subsec: shape}
\begin{teorema}\label{thm: regequi}
	Let $F^s_\e(x)$ denote the equilibrium solution of equation \eqref{DFFP}. Then $F^s_\e$ satisfies the following regularity properties:
	\begin{enumerate}[(i)]
		\item \textbf{Weak regularity:}
		\begin{enumerate}[(a)]
			\item $F^s_\e \in H^m(\mathbb{R}^d)$ for all $0<m < \dfrac{1}{\e^{2s}} - \dfrac{d}{2}$,
			\item $F^s_\e \notin H^m(\mathbb{R}^d)$ for all $m > \dfrac{1}{\e^{2s}} - \dfrac{d}{2}$.
		\end{enumerate}
		
		\item \textbf{Classical regularity:} 
		\begin{enumerate}[(a)]
			\item $F^s_\e \in \mathcal{C}^{\fl{m},m-\fl{m}}(\mathbb{R}^d)$ for all $0<m < \dfrac{1}{\e^{2s}}-d$,
			\item $F^s_\e \notin \mathcal{C}^{\fl{m},m-\fl{m}}(\mathbb{R}^d)$ for all $m > \dfrac{1}{\e^{2s}}$.
		\end{enumerate}

	\end{enumerate}
\end{teorema}
\begin{proof}
We only present the proof of  $(i)$, as the proof of  $(ii)$ follows by a similar argument. The details of $(ii)$ can be found, with minor modifications  in \citep[Sec.~5]{canizotassi2024}, to incorporate the exponent $2s$, and the fractional (H\"older) regularity via Besov space embedding, which analogously follows from the decay of the Fourier transform.

Recall that the Fourier transform of the equilibrium is $$
\hat{F}^s_\e(\xi)=\exp\Big(\frac{1}{\e^{2s}}\int_0^\infty\zeta_\e^s(e^{-z}\xi)\d z\Big)
$$
where $\zeta_\e^s(\xi)=\hat{J_\e^s}(\xi)-1$ is the Fourier symbol of $\mathcal{A}^s_\e$.

Let us assume that $\zeta^s_\e$ is radially symmetric.
 If not, define
$\underline{\zeta}^s_\e$ and $\overline\zeta^s_\e$ by
\begin{equation*}
	\underline{\zeta}^s_\e(z)=\min_{|y|=|z|}\zeta^s_\e(y)\qquad \overline{\zeta}^s_\e(z)=\max_{|y|=|z|}\zeta^\e_s(y),
\end{equation*}
which are radially symmetric, and it suffices to carry out the analysis using them in place of $\zeta^s_\e$.
After a change of variable $y=e^{-z}|\xi|$, we obtain
$$
\hat{F}^s_\e(\xi)=\exp\Big(\frac{1}{\e^{2s}}\int_0^{|\xi|}\frac{\zeta_\e^s(y)}{y}\d y\Big).
$$
Since $\zeta_\e^s(\xi)=-\e^{2s}|\xi|^{2s}+R^s(\e\xi)$, close to $\xi=0$, one has, as expected
$$
\hat{F}_\e^s(\xi)\approx e^{-\frac{|\xi|^{2s}}{2s}}.
$$
Notice that since $J_\e^s\in L^1$, its Fourier transform is continuous and, by Riemann-Lebesgue Lemma, $J_\e^s(y)\to0$ as $|y|\to\infty$.
By Plancherel 
\begin{multline*}
\norm{F_\e^s}^2_{H^m}=(2\pi)^d\ird (1+|\xi|^2)^m|\hat{F}^s_\e(\xi)|^2\d \xi
\\
=(2\pi)^d\Big(\int_{|\xi|\le R} (1+|\xi|^2)^m|\hat{F}^s_\e(\xi)|^2\d \xi+\int_{|\xi|> R} (1+|\xi|^2)^m|\hat{F}^s_\e(\xi)|^2\d \xi\Big)
\end{multline*}
The "low-frequency" piece is always finite, so the question reduces to whether 
$$
\int_{|\xi|>R}|\xi|^{2m}|\hat{F}_\e^s|^2\d \xi
$$
converges or diverges.
\begin{enumerate}[(a)]
	\item Assume $m<\dfrac{1}{\e^{2s}}-\dfrac{d}{2}$ and choose $\delta>0$ so small  that 
	\begin{equation}\label{eq: deltabound}
		2m-\frac{2(1-\delta)}{\e^{2s}}+d<
	0
\end{equation}
Since $\hat{J}_\e^s(y)\to 0$ as $|y|$ gets larger, we
 define $y_\delta$  such that
 $$\hat{J}^s_\e(y)\le\delta\qquad\text{for all } |y|\ge y_\delta$$
 Thus taking $R>y_\delta$ (such that $|\xi|>y_\delta$), we can write
 $$
 \int_0^{|\xi|}\frac{\zeta_\e^s (y)}{y}\d y=\int_0^{y_\delta}\frac{\zeta_\e^s (y)}{y}\d y
+\int_{y_\delta}^{|\xi|}\frac{\zeta_\e^s (y)}{y}\d y \le C-(1-\delta)(\log|\xi|-\log{y_\delta})$$
which implies
$$
|\hat{F}_\e^s(\xi)|\le \tilde{C}|\xi|^{-\frac{1-\delta}{\e^{2s}}}
$$
Thus
$$\int_{|\xi|>R}|\xi|^{2m}|\hat{F}_\e^s|^2\d \xi\le (2\pi)^{-d}\tilde{C}\int_{|\xi|>R}|\xi|^{2m-\frac{2(1-\delta)}{\e^{2s}}}\d \xi= (2\pi)^{-d}\tilde{C}\int_R^\infty r^{2m-\frac{2(1-\delta)}{\e^{2s}}+d-1}\d r<\infty,
$$
 by  \eqref{eq: deltabound}.
 
 \item Assume $m>\dfrac{1}{\e^{2s}}-\dfrac{d}{2}$ Choose $\delta>0$ so small that
\begin{equation}\label{eq: deltabound2}
2m-\frac{2(1+\delta)}{\e^{2s}}+d>0,\end{equation}
and again define $y_\delta$ such that 
$$
\hat{J}^s(y)\ge-\delta\qquad\text{for all }|y|\ge y_\delta.
$$
Proceeding as before
$$
|\hat{F}_\e^s(\xi)|\ge |\xi|^{-\frac{1+\delta}{\e^{2s}}},
$$
and hence,

$$\int_{|\xi|>R}|\xi|^{2m}|\hat{F}_\e^s|^2\d \xi\ge (2\pi)^{-d}\tilde{C}\int_{|\xi|>R}|\xi|^{2m-\frac{2(1+\delta)}{\e^{2s}}}\d \xi= (2\pi)^{-d}\tilde{C}\int_R^\infty r^{2m-\frac{2(1+\delta)}{\e^{2s}}+d-1}\d r=\infty,
$$
by \eqref{eq: deltabound2}.

\end{enumerate}

\end{proof}
As in the case $s=1$, we have the following theorem for which we do not give the proof but refer again to \citep[Sec.5]{canizotassi2024}
\begin{teorema}
    Let $\Omega\subset\subset \R^d$ such that $0\notin\Omega$. Suppose (at least) one of the two following conditions holds
    \begin{enumerate}[(a)]
        \item $d=1$
        \item $J$ is radially symmetric.       
    \end{enumerate}
     then $F_\e\in \mathcal{C}^\infty(\R^d\setminus \{0\}).$
\end{teorema}


\section{Non singular to singular limit: $\e\to 0$}\label{sec: epsilon}
We now gather results on the convergence as $\e\to 0$ in $L^2$ spaces and then we will transfer the results to the $L^1_k$ spaces thanks to the following lemma

\begin{lemma}
	[Interpolation]
	Let $0<k<m$, $f\in L^2\cap L^1_m$.
	Then $$
	\norm{f}_{L^1_k}\le \norm{f}_{L^2}^\theta\norm{f}_{L^1_m}^{1-\theta}
	$$
	with \begin{equation}\label{eq: theta}
		\theta=\frac{m-k}{m+d/2}\end{equation}
\end{lemma}
\begin{proof}
	For some $R$ to determine
	\begin{multline*}
		\norm{f}_{L^1_k}=\int_{\ap{x}\le R}\ap{x}^k|f|\d x+\int_{\ap{x}>R}\ap{x}^k|f|\d x\le \ap{R}^k\int_{\ap{x}\le R}|f|\d x+
		\int_{\ap{x}>R}
		\ap{x}^{k-m}\ap{x}^m|f|\d x\\
		\le \ap{R}^k\int|f|\chi(\ap{x}\le R)+\ap{R}^{k-m}\ird\ap{x}^m|f|\d x\le \ap{R}^k\norm{f}_{L^2}\Big(\int_{\ap{x}\le R}\d x\Big)^{1/2}+\ap{R}^{k-m}\norm{f}_{L^1_m}\\
		\lesssim R^{k+d/2}\norm{f}_{L^2}+R^{k-m}\norm{f}_{L^1_m}
		\end{multline*}
		Choosing $$R=\Big(\frac{\norm{f}_{L^1_m}}{\norm{f}_{L^2}}\Big)^{\frac{1}{m+d/2}}$$
		one concludes the proof
	\end{proof}

\begin{proposizione}
    \label{prop: consistencyEPS} Let $(J^s)$ be a family of kernels satisfying Hypothesis \ref{ass: sunif} and
    let $A_\e^s$ be the nonlocal operator defined in \eqref{eq: diffoper}. Take  $\f\in \dot{H}^{\eta}\cap L^{1}_M$ with $\eta>2s+\delta$, $1\ge M>m>k$, the homogeneous (fractional) Sobolev space.
    Then,
    \begin{equation}\label{eq: consL2}
        \norm{\mathcal{A}^s_\e\f+(-\Delta)^s\f}_{L^2}\le C\e^{\delta}\norm{\f}_{\dot{H}^{\eta}}
    \end{equation}
    and 
    \begin{equation}\label{eq: consL1k}
    \norm{\mathcal{A}^s_\e\f+(-\Delta)^s\f}_{L^1_k}\le C\e^{\theta\delta}\norm{\f}^\theta_{\dot{H}^{\eta}}\norm{\f}_{W^{\eta,1}_m}^{1-\theta}
    \end{equation} 
    where $\theta\in(0,1)$ is  as in \ref{eq: theta}).
the constant $C$ is independent of $s$.
    \end{proposizione}
    \begin{proof}    
    First observe that by interpolation $f\in W^{\eta,1}_m$.
    Write the reminder as
    $R(\xi):=J(\xi)-1+|\xi|^{2s}$ is such that \begin{equation}\label{eq: remainder}
    	|R(\xi)|\le C_0|\xi|^{2s+\delta}\quad \text{ for all $\xi$}\end{equation}
    Indeed by hypothesis \ref{ass: sunif}, there exists  $r>0$ such that $|R(\xi)|\le C_1|\xi|^{2s+\delta}$, for all $|\xi|<r$.
    However the local bound is in reality a global one: we can see that for all $|\xi|\ge r$
    $$
    \frac{|R(\xi)|}{|\xi|^{2s+\delta}}=\frac{\abs{J(\xi)-1+|\xi|^{2s}}}{|\xi|^{2s+\delta}}\le \frac{|J|}{|\xi|^{2s+\delta}}+\frac{1}{|\xi|^{2s+\delta}}+|\xi|^{-\delta}\le 2r^{-2s-\delta}+r^{-\delta}:=C_2
    $$
    By Plancherel
    \begin{equation}\label{eq: plancherel}
    	\norm{\mathcal{A}^s_\e\f+(-\Delta)^s\f}_{L^2}=(2\pi)^{-d}\norm{\mathcal{F}[\mathcal{A}^s_\e\f+(-\Delta)^s\f]}_{L^2}=(2\pi)^{-d}\Big(\int|\mathfrak{m}_\e(\xi)|^2|\hat{\f}(\xi)|^2\d \xi\Big)^{1/2}
    \end{equation}
    where $\mathfrak{m}_\e(\xi)$ is the  Fourier multiplier of $A_\e^s+(-\Delta)^s $ defined as  $$\mathfrak{m}_\e(\xi)=\frac{1}{\e^{2s}}(\hat{J}(\e\xi)-1)+|\xi|^{2s}$$  
    Taking into account \eqref{eq: remainder}
    \begin{equation*}
    	|\mathfrak{m}_\e(\xi)|
    	=\abs{\frac{1}{\e^{2s}}[1-|\e\xi|^{2s}+R(\e\xi)-1]+|\xi|^{2s}}=\frac{1}{\e^{2s}}|R(\e\xi)|\le C_0 \e^{\delta}|\xi|^{2s+\delta}
    \end{equation*}
     Therefore plugging it into \eqref{eq: plancherel} 
    we prove \eqref{eq: consL2}
    \begin{multline*}\label{eq: plancherel2}
    	\norm{\mathcal{A}^s_\e\f+(-\Delta)^s\f}_{L^2}\le(2\pi)^{-d}\Big(\int|\mathfrak{m(\xi)}_\e|^2|\hat{\f}(\xi)|^2\d \xi\Big)^{1/2}\\
    	\le C_0\e^{\delta}(2\pi)^{-d}\Big(\int|\xi|^{2(2s+\delta)}|\hat{\f}(\xi)|^2\d \xi\Big)^{1/2}\le C\e^{\delta}\norm{\f}_{\dot{H}^{2s+\delta}}\le    C\e^{\delta}\norm{\f}_{\dot{H}^{\eta}} 	
    \end{multline*}
    Via interpolation inequality  one obtains \eqref{eq: consL1k}.

     \end{proof}
  
     \begin{osservazione}\label{rmk: smalldelta}
     We decided to prove the consistency in $L^2$, because the Fourier representation is particularly well-suited for capturing the behaviour of the operator as $s\to 1$; the trade–off is the loss in the $\e$ exponent, when returning to the   $L^1_k$–space.
     
     On the other hand, following the proof in \citep{canizotassi2024}, a direct analysis in $L^1_k$, only requiring initial data  in $W^{\eta,1}_k$, yields a convergence rate of $\e^{\gamma_2}$, where 
     $\gamma_{2}=\min\{\delta, 2(1-s)\}$     
     By combining this two approaches, we can improve the exponent in Proposition \ref{prop: consistencyEPS} to $\gamma=\max\{\theta\delta,\gamma_2\}$ 
     \end{osservazione}

\begin{proposizione}\label{prop: stabilityEPS}
Let $(J^s)$ be a family of kernels satisfying Hypothesis \ref{ass: sunif}.
	Let $u\in \mathcal{C}([0,\infty);L^1_k)$ be the solution of $$\partial_t u=L_\e^s u$$
	with initial data $u(0)=u_0$.
Consider the functions $v, h \in \mathcal{C}([0,\infty);L^1_k)$, where 
      $v$ is the solution in the mild sense of $$
     \partial_tv=L_\e^sv+h(t)$$
     with the  initial data $v(0)=u_0$, and $h$ a  perturbation.
     Then for all $t\ge 0$,
      $$
     \norm{u(t)-v(t)}_{L^1_k}\le 2\frac{C}{\lambda}\int_0^t \norm{h(\tau)}_{L^1_k}\d \tau$$
     with constants $C,\lambda$ from Theorem \ref{thm: lyapunov}, depending only on $d$ and $k$
\end{proposizione}
\begin{proof}
	This is a simple consequence of Lyapunov's estimate. We have proven that there exist constants $C>1, $ $\lambda>0$ such that $$
	\norm{e^{t L_\e^s}f}\le e^{-\lambda t}\norm{f}_{L^1_k}+\frac{C}{\lambda}(1-e^{-\lambda t})\norm{f}_{L^1}\le 2\frac{C}{\lambda}\norm{f}_{L^1_k}
	$$

Now let us define $\psi=v-u$ that solves
	\begin{equation}
		\begin{cases}
			\partial_t\psi=L_\e^s\psi+h(t)\\
			\psi(0)=0
		\end{cases}
	\end{equation}
	Applying Duhamel's formula
	$$
	\norm{\psi(t)}_{L^1_k}=\norm{\int_0^t e^{(t-\tau)L_\e^s}h(\tau)\d \tau}_{L^1_k}\le \int_0^t\norm{e^{(t-\tau)L_\e^s}h(\tau)}_{L^1_k}\d \tau\le 2\frac{C}{\lambda}\int_0^t \norm{h(\tau)}_{L^1_k}\d \tau
	$$
\end{proof}

\begin{teorema}\label{thm: convfinitetimeEPS}   
	Let $(J^s)$ be a family of kernels satisfying  Hypothesis \ref{ass: sunif}. Let $u_\e^s$ be a solution of \eqref{DFFP} with initial data $u_0\in\dot{H}^{\eta}\cap L^1_M$ with $1\ge M>m>k$, and $v^s$ a solution of the fractional Fokker--Planck equation \eqref{eq: FFP} with the same initial data.
    Then for all $t\in[0,T^*],$ there exists a constant $C$ independent of $s$, such that 
    \begin{equation*}
    	\norm{u_\e^s(t)-v^s(t)}_{\aux}\le CT^*\e^{\theta\delta}
    \end{equation*}
\end{teorema}
\begin{proof}
	By Theorem \ref{thm: propFFP} $(i)$, the solution of \eqref{eq: FFP}  is immediately $H^{\eta}$. Hence $v^s$ satisfies the hypotheses of Proposition \ref{prop: consistencyEPS} uniformly in $[0,T^*]$.
	Rewriting
	$$L_0^s v^s=L_\e^s v^s+(L_0^s v^s-L_\e^s v^s):=L_\e^s v+h(t)$$
	Then applying first Proposition \ref{prop: stabilityEPS} and then Proposition \ref{prop: consistencyEPS}
	\begin{equation*}
	\norm{u_\e^s(t)-v^s(t)}_\aux\le 2\frac{C_L}{\lambda_L}\int_0^t\norm{L_0^s v^s-L_\e^sv^s}_\aux\le C_1\frac{C_L}{\lambda_L} T^*\e^{\theta\delta} 
\end{equation*}
\end{proof}

\begin{teorema}\label{thm: conveqEPS}
	Let $(J^s)$ be a family of kernels satisfying Hypothesis \ref{ass: sunif}, and let $F^s_\e$ and $F^s$ be the steady states of respectively \eqref{DFFP} and \eqref{eq: FFP}.
	Then there exists a constant $C$ independent of $s$, such that 
    $$
	\norm{F_\e^s-F^s}_{\aux}\le C\e ^{\esp}	
    $$

\end{teorema}
\begin{proof}
	By the Feller-Miyadera-Phillips version of the Hille-Yosida
	 theorem \citep[Part 2, Thm. 3.8]{EngelNagel2001}, the estimate
	\begin{equation}\label{eq: HY1}
	\norm{S_t} \le K e^{-\lambda t}
	\qquad(\gamma>0, K\ge1, t\ge0)
	\end{equation}
	for a strongly continuous semigroup $(S_t)$ on a Banach space $\mathcal Y$ is equivalent to the fact that for every $\gamma>-\lambda$,
	$\gamma$ lies in the resolvent set $\rho(L_\e)$ of its generator $(L_\e^s,\mathcal D(L_\e^s))$ and, moreover,
	$$
	\norm{R(\gamma,L_\e^s)^n}
	 \le 
	\frac{K}{(\lambda+\gamma)^n}
	\quad\text{for all }n\in\mathbb N,
	$$
	where $R(\gamma,L_\e^s)=(\gamma - L_\e^s)^{-1}$ denotes the resolvent operator.
	From Theorem \ref{thm: main1},
	the condition \eqref{eq: HY1} is satisfied in the Banach space $\mathcal{Y}:=\{f\in L^1_k: \int f=0\}.$
	Choosing $n=1, \gamma=0>-\lambda$, it follows that $L_\e^s$ has an inverse in $\mathcal{Y}\subset L^1_k$, with 
	$\norm{(L_\e^s)^{-1}}_{L^1_k}\le K/\lambda$.
	Since
	 $L_\e^sF_\e^s=L_0^sF^s=0$,
	\begin{multline*}
	\norm{F_\e^s-F^s}_{L^1_k}=\norm{L_\e^s (L_\e^s)^{-1}(F_\e^s-F^s)}_{\mathcal{Y}}\le \norm{(L_\e^s)^{-1}}_{\mathcal{Y}}\norm{L_\e^s(F_\e^s-F^s)}_{\mathcal{Y}}\\
	\le \frac{K}{\lambda}\norm{L_\e^s(F_\e^s-F^s)}_{L^1_k}\le \frac{K}{\lambda}\norm{(L_\e^s-L^s_0)F^s}_{L^1_k}
\end{multline*}
The stable law  $F^s=G^{s,(2s)^{-1}}$ is in all $H^\eta$ for all $\eta>0$ 
Indeed, 
$$\norm{F^s}_{H^{k}}^2=\int\ap{\xi}^{2k} e^{-2\frac{|\xi|^{2s}}{2s}}\d \xi<\infty.$$
 Moreover, $F^s\in L^1_M$ for all $M<1$: by consistency \ref{prop: consistencyEPS}, we can conclude
$$\norm{(L_\e^s-L_0^s)F^s}_{L^1_k}\le C\e^{\theta\delta}$$

\end{proof}

\begin{teorema}[Convergence for large times]\label{thm: convunifEPS}
	Let $(J^s)$ be a family of kernels satisfying Hypothesis \ref{ass: sunif},and let $u_\e^s$ be the solution of \eqref{DFFP} with initial data $u_0$ while denote by $v^s$ the solution of the fractional Fokker-Planck \eqref{eq: FFP} with initial data $v_0$, .
	Then for all $\e\in(0,1]$ there exists $T>0$ such that for all $t\ge T$ there exists a constant independent of $s$, such that 
	$$
	\norm{u^s_\e(t)-v^s(t)}_\aux\le C\e^{\theta\delta}
	$$
\end{teorema}
\begin{proof}

We decompose the error into three parts: the relaxation of $u_\e ^s$ to its equilibrium, the closeness of the two equilibria, and the relaxation of the limit solution $v^s$ to its own equilibrium, which is a well-known fact.

Applying Theorem \ref{thm: main1}, the exponential convergence result of equation \eqref{eq: SG-FFP}, and Theorem \ref{thm: conveqEPS} 
	\begin{multline*}
			\norm{u^s_\e(t)-v^s(t)}_\aux\le \norm{u^s_\e(t)-F^s_\e}_\aux+\norm{v^s(t)-F^s}_\aux+\norm{F^s_\e-F^s}_\aux\\
			\le C_1e^{-\lambda_1 t}\norm{u_0-F_\e^s}_\aux+C_2 e^{-\lambda_2 t}\norm{v_0-F^s}_\aux+ \tilde{C}\e^{\esp}\le C\e^{\esp}
	\end{multline*}
	for $t$ large enough
\end{proof}

\begin{osservazione}
	The rate of convergence can be sharpen as shown in \ref{rmk: smalldelta}
\end{osservazione}

Gathering the results of Theorem \ref{thm: convfinitetimeEPS} and \ref{thm: convunifEPS}, we prove Theorem \ref{thm: main2}.

    
\section{Long--range to short--range limit: $s\to 1$}\label{sec: s}
Here we need the assumption on the $J^s$ that converge to $J$.
Let us recall the definition of the Fourier remainder  $$R^s(\xi):=\hat{J}^s(\xi)-1+|\xi|^{2s}.$$

For the whole chapter we will assume Hypothesis \ref{ass: remainderest}
which is equivalent to 
$$\abs{\frac{R^s(\xi)}{|\xi|^{2s}}-\frac{R^1(\xi)}{|\xi|^2}}\le C(1-s)|\xi|^{\delta}\max\{1,\log(|\xi|)\}.$$

\begin{proposizione}
    \label{prop: consistencyS}Let 
     $\f\in \dot{H}^{\eta}\cap L^1_M$ for some $\eta>2+\delta$, $1\le M< m< k$. Then 
      
     $$
     	\norm{\mathcal{A}^s_\e\f-\mathcal{A}^1_\e\f}_{L^2}\le C(1-s)\norm{\f}_{\dot{H}^{\eta}}[1+\e^\delta|\log(|\e|)|]
     $$
     and  
     $$
     \norm{\mathcal{A}^s_\e\f-\mathcal{A}^1_\e\f}_{L^1_k}\le C(1-s)^\theta\norm{\f}_{\dot{H}^{\eta}}^\theta[1+\e^\delta|\log(|\e|)|]^{\theta}
     $$
     where $\theta=\frac{m-k}{m+d/2}$
     
    \end{proposizione}
    \begin{proof}    
    	Proceeding similarly to Proposition \ref{prop: consistencyEPS}, 
    	we have 
    	$$
   		|R^s(\xi)|\le C_1|\xi|^{2s+\delta}\qquad|R^1(\xi)|\le C_2|\xi|^{2+\delta}$$
   		Again, applying Plancherel's theorem
   		  \begin{equation}\label{eq: plancherelS}
   			\norm{\mathcal{A}^s_\e\f-\mathcal{A}^1_\e\f}_{L^2}=(2\pi)^{-d}\norm{\mathcal{F}[\mathcal{A}^s_\e\f-\mathcal{A}^1_\e\f]}_{L^2}=(2\pi)^{-d}\Big(\int|\mathfrak{m(\xi)}_\e|^2|\hat{\f}(\xi)|^2\d \xi\Big)^{1/2}
   		\end{equation}
   		 where $\mathfrak{m}_\e(\xi)$ is the  Fourier multiplier of $A^s_\e-\mathcal{A}^1_\e $ defined as  $$\mathfrak{m}_\e(\xi)=\frac{1}{\e^{2s}}(\hat{J^s}(\e\xi)-1)-\frac{1}{\e^2}(\hat J(\e\xi)-1)$$  
   	Then 
   		\begin{multline*}
   			|\mathfrak{m}_\e(\xi)|
   			=\abs{|\xi|^{2s}\frac{\hat{J^s}(\e\xi)-1}{|\e\xi|^{2s}}-|\xi|^{2}\frac{\hat{J}(\e\xi)-1}{|\e\xi|^{2}}}
   			\\
   			\le\abs{|\xi|^{2s}-|\xi|^2}\abs{\frac{\hat{J^s}(\e\xi)-1}{|\xi|^{2s}}}+|\xi|^{2}\abs{\frac{\hat{J^s}(\e\xi)-1}{|\e\xi|^{2s}}-\frac{\hat{J}(\e\xi)-1}{|\e\xi|^{2}}}
   			\\
   			\le |\xi|^{2s}\abs{1-|\xi|^{2(1-s)}}\Big(1+\frac{|R^s(\xi)|}{|\xi|^{2s}}\Big)+\mathfrak{C}|\xi|^{2}(1-s)|\e\xi|^{\delta}\abs{\log(\e|\xi|)}
   			\\
   			\le 2|\xi|^{2s}\abs{\log(|\xi|)}(1+|\xi|^{2(1-s)})(1+C_1|\xi|^{\delta})+\mathfrak{C}|\xi|^{2}(1-s)|\e\xi|^{\delta}\big(\abs{\log(|\xi|)}+\abs{\log(|\e|)}\big)
   			\\
   			\le C(1-s)|\xi|^{2+\delta}\abs{\log(|\xi|)}[1+\tilde{C}\e^{\delta}\abs{\log(|\e|)}]
   		\end{multline*}
   		Thus,
   		 \begin{multline*}
   			\norm{\mathcal{A}^s_\e\f-\mathcal{A}^1_\e\f}_{L^2}\le (2\pi)^{-d}\Big(\int\abs{C(1-s)|\xi|^{2+\delta}\abs{\log(|\xi|)}[1+\tilde{C}\e^{\delta}\abs{\log(|\e|)}]}^2|\hat{\f}(\xi)|^2\d \xi\Big)^{1/2}
   			\\
   			\le C(1-s)|\xi|^{2+\delta}[1+\tilde{C}\e^{\delta}\abs{\log(|\e|)}]\Big(\int |\xi|^{2+\delta}\abs{\log(|\xi|)}|\hat{\f}(\xi)|^2\d \xi\Big)^{1/2}
   			\\
   			\lesssim (1-s)\norm{f}_{\dot{H}^\eta}[1+\e^\delta|\log(|\e|)|].
   		\end{multline*}
   		
    \end{proof}
    
    Due to the lack of regularisation effects,  at short time we must assume additional regularity of the initial data,

\begin{teorema}\label{thm: convfinitetimeS}
	Let $(J^s)$ be a family of kernels satisfying Hypotheses \ref{ass: sunif} and \ref{ass: remainderest}. Let  $u_0\in L^{1}_M\cap H^{\eta}$  with $1\ge M>m>k$ and $\eta>2+\delta$, and let $u_\e^s$ and $u_\e$ be solutions of \eqref{DFFP} and \eqref{eq: NLFP} respectively,  with the same initial data $u_0$.
	Then for all $t\in[0,T^*],$ there exists $C=C(d,k,M,m, u_0)$
	\begin{equation*}
		\norm{u_\e^s(t)-u_\e(t)}_{\aux}\le C(1-s)^\theta(1+\e^{\delta\theta}\abs{\log|\e|})
	\end{equation*}
\end{teorema}
\begin{proof}
 First notice that by interpolation $f\in W^{\eta,1}_m$.
	By Theorem \ref{thm: propFFP} $(i)$, the solution of \eqref{eq: FFP} with initial data $u_0\in L^2$ is immediately $H^{2s+\delta}$. Hence $u_\e$ satisfies the  hypotheses of Proposition \ref{prop: consistencyEPS} uniformly in $[0,T^*]$.
	Rewriting
	$$L_0^s u_\e=L_\e^s u_\e+(L_\e u_\e-L_\e^s u_\e):=L_\e^s u_\e+h(t)$$
	Then applying first Proposition \ref{prop: stabilityEPS} and then Proposition \ref{prop: consistencyS}
	\begin{equation*}
		\norm{u_\e^s(t)-u_\e(t)}_\aux\le 2\frac{C_L}{\lambda_L}\int_0^t\norm{L_\e u_\e-L_\e^su_\e}_\aux\le C_1\frac{C_L}{\lambda_L} T^*(1-s)^\theta(1+\e^{\theta\delta}\abs{\log|\e|})
	\end{equation*}
\end{proof}
Similarly to the previous case, we show the convergence of the equilibria.

\begin{teorema}\label{thm: conveqS}
Let $(J^s)$ be a family of kernels satisfying Hypotheses \ref{ass: sunif} and \ref{ass: remainderest}, and let $F^s_\e$ and $F^s$ be the steady states of respectively \eqref{DFFP} and \eqref{eq: FFP}.
Then for all $\e\in\big(0,\sqrt{\dfrac{2}{4+2\delta+d}}\Big)$, $$
\norm{F_\e^s-F^\e}_{\aux}\le C(1-s)^\theta (1+\e^{\delta\theta}\abs{\log|\e|}).
$$
\end{teorema}
\begin{proof}
We have already proven in Theorem \ref{thm: conveqEPS} that $L_\e^s$ has an inverse in the Banach space $\mathcal{Y}:=\{f\in L^1_k: \int f=0\},$ with 
$\norm{(L_\e^s)^{-1}}_{L^1_k}\le K/\lambda$.
Since
$L_\e^sF_\e^s=L_\e F_\e=0$,
\begin{multline*}
	\norm{F_\e^s-F_\e}_{L^1_k}=\norm{L_\e^s (L_\e^s)^{-1}(F_\e^s-F_\e)}_{\mathcal{Y}}\le \norm{(L_\e^s)^{-1}}_{\mathcal{Y}}\norm{L_\e^s(F_\e^s-F_\e)}_{\mathcal{Y}}\\
	\le \frac{K}{\lambda}\norm{L_\e^s(F_\e^s-F_\e)}_{L^1_k}\le \frac{K}{\lambda}\norm{(L_\e^s-L_\e)F_\e}_{L^1_k}
\end{multline*}
From Theorem \ref{thm: regequi} $(i)\text{-}a$ and the fact that $\e\in\Big(0,\sqrt{\dfrac{2}{4+2\delta+d}}\Big)$, we have that $F_\e\in H^{\eta}$ with $\eta>2+\delta$.
Moreover, from \citep[Sec. 4]{canizotassi2024}, $F_\e$ has at least exponential tails, which implies that $F_\e\in L^1_M$ for all $M$.

We can then apply  \ref{prop: consistencyS}, and conclude
$$\norm{(L_\e^s-L_\e)F_\e}_{L^1_k}\le C(1-s)^\theta (1+\e^{\delta\theta}\abs{\log|\e|}).$$
\end{proof}Asymptotic relaxation ensures that the solution eventually 'forgets' its initial state. Given the regularity of the equilibrium for small $\e$, we can  apply the consistency estimate.
\begin{teorema}[Convergence for large times]\label{thm: convunifS}
	Let $(J^s)$ be a family of kernels satisfying Hypotheses \ref{ass: sunif}, and \ref{ass: remainderest} and let $u_\e^s$ be the solution of \eqref{DFFP} with initial data $u_0$ while denote by $u_\e$ the solution of the nonlocal Fokker-Planck \eqref{eq: NLFP} with initial data $\bar{u}_0$.
	Then for all $\e\in\big(0,\sqrt{\dfrac{2}{4+2\delta+d}}\Big)$ there exists $T^*>0$ such that for all $t\ge T^*$
	$$
	\norm{u^s_\e(t)-u_\e(t)}_\aux\le C (1-s)^\theta (1+\e^{\theta\delta}\abs{\log|\e|}).
	$$
\end{teorema}
\begin{proof}Applying Theorem \ref{thm: main1}, the exponential convergence result of \citep[Thm 1]{canizotassi2024} and Theorem \ref{thm: conveqS}, we have
	\begin{multline*}
		\norm{u^s_\e(t)-u_\e(t)}_\aux\le \norm{u^s_\e(t)-F^s_\e}_\aux+\norm{u_\e-F_\e}_\aux+\norm{F^s_\e-F_\e}_\aux\\
		\le C_1e^{-\lambda_1 t}\norm{u_0-F_\e^s}_\aux+C_2 e^{-\lambda_2 t}\norm{\bar{u}_0-F_\e}_\aux+ \tilde{C}(1-s)^\theta (1+\e^{\theta\delta}\abs{\log|\e|})\\
		\le {C}(1-s)^\theta (1+\e^{\theta\delta}\abs{\log|\e|}),
	\end{multline*}
	for $t$ large enough.
\end{proof}
Combining Theorem \ref{thm: convunifS} and \ref{thm: convfinitetimeS}, we conclude the proof of Theorem \ref{thm: main3}.

\section*{Acknowledgments}
I would like to thank José A. Cañizo for suggesting the problem and for the  many helpful discussions.

The project that gave rise to these results received the support of a
fellowship from the ``la Caixa'' Foundation (ID 100010434). The
fellowship code is LCF/BQ/DI22/11940032. The author acknowledges
support from the ``Maria de Maeztu'' Excellence Unit IMAG, reference
CEX2020-001105-M, funded by MCIN/AEI/10.13039/501100011033/. He was
also supported by grant PID2023-151625NB-I00 and the research network
RED2022-134784-T from MCIN/AEI/10.13039/501100011033/.
\addcontentsline{toc}{section}{References}
\bibliography{bibliography}

@Article{Rey2013,
  author               = {Rey, Thomas and Toscani, Giuseppe},
  title                = {{Large-Time} Behavior of the Solutions to {Rosenau-Type} Approximations to the Heat Equation},
  doi                  = {10.1137/120876290},
  issn                 = {0036-1399},
  number               = {4},
  pages                = {1416--1438},
  volume               = {73},
  journal              = {SIAM Journal on Applied Mathematics},
  month                = jan,
  year                 = {2013},
}

@Book{AndreuVaillo2010,
  author               = {Andreu-Vaillo, Fuensanta and Maz\'{o}n, José M. and Rossi, Julio D. and Toledo-Melero, Juli\'{a}n J.},
  title                = {Nonlocal diffusion problems},
  isbn                 = {9780821852309},
  publisher            = {American Mathematical Society ; Real Sociedad Matem\'{a}tica Española},
  url                  = {http://www.worldcat.org/isbn/9780821852309},
  year                 = {2010},
}

@InProceedings{Hairer2011,
  author               = {Hairer, Martin and Mattingly, Jonathan C.},
  booktitle            = {Seminar on Stochastic Analysis, Random Fields and Applications VI},
  title                = {Yet Another Look at {Harris}' Ergodic Theorem for {Markov} Chains},
  chapter              = {7},
  doi                  = {10.1007/978-3-0348-0021-1\_7},
  editor               = {Dalang, Robert and Dozzi, Marco and Russo, Francesco},
  eprint               = {0810.2777},
  isbn                 = {978-3-0348-0020-4},
  pages                = {109--117},
  publisher            = {Springer Basel},
  url                  = {http://dx.doi.org/10.1007/978-3-0348-0021-1\_7},
  address              = {Basel},
  archiveprefix        = {arXiv},
  day                  = {15},
  month                = oct,
  year                 = {2011},
}

@Book{Meyn2010,
  author               = {Meyn, Sean P. and Tweedie, Richard L.},
  title                = {Markov chains and stochastic stability},
  isbn                 = {9780521731829},
  publisher            = {Cambridge University Press},
  url                  = {http://www.worldcat.org/isbn/9780521731829},
  year                 = {2010},
}

@Article{Carrillo2007,
  author               = {Carrillo, J. A. and Toscani, G.},
  title                = {{Contractive probability metrics and asymptotic behavior of dissipative kinetic equations}},
  pages                = {75--198},
  volume               = {6},
  journal              = {Riv. Mat. Univ. Parma (7)},
  mrnumber             = {2355628 (2008i:82094)},
  year                 = {2007},
}

@Article{Gualdani2018,
  author               = {Gualdani, Maria P. and Mischler, Stéphane and Mouhot, Clément},
  title                = {{Factorization for non-symmetric operators and exponential H-theorem}},
  eprint               = {1006.5523},
  eprinttype           = {arxiv},
  note                 = {To appear},
  archiveprefix        = {arXiv},
  day                  = {29},
  journal              = {Mémoires de la Société Mathématique de France},
  year                 = {2018},
}

@Article{Canizo2018b,
  author    = {Jos{\'{e}} A. Cañizo and Jos{\'{e}} A. Carrillo and Manuel P{\'{a}}jaro},
  title     = {Exponential equilibration of genetic circuits using entropy methods},
  doi       = {10.1007/s00285-018-1277-z},
  number    = {1-2},
  pages     = {373--411},
  volume    = {78},
  journal   = {Journal of Mathematical Biology},
  month     = {aug},
  publisher = {Springer Science and Business Media {LLC}},
  year      = {2018},
}

@article{berry_1941,
	title = {{The} {Accuracy} of  the {Gaussian} {Approximation} to the {Sum} of {Independent} {Variates}},
	author = {Berry, Andrew C},
        journal ={Transactions of the american mathematical society},
        year = {1941},
	langid = {english},
}

@article{bentkus_2005,
author = {Bentkus, V.},
title = {A  {Lyapunov-type} {Bound} in $\textbf{R}^d$},
journal = {Theory of Probability \& Its Applications},
volume = {49},
number = {2},
pages = {311-323},
year = {2005},
doi = {10.1137/S0040585X97981123},

URL = { 
    
        https://doi.org/10.1137/S0040585X97981123
    
    

},
eprint = { 
    
        https://doi.org/10.1137/S0040585X97981123
    
    

}
}

@article{raic_2019,
	title = {A {M}ultivariate {B}erry-{E}sseen {T}heorem with explicit constants},
 journal ={Bernoulli},
	volume = {25},
        year = {2019},
	issn = {1350-7265},
	url = {http://arxiv.org/abs/1802.06475},
	doi = {10.3150/18-BEJ1072},
	number = {4},
	journaltitle = {Bernoulli},
	shortjournal = {Bernoulli},
	author = {Raič, Martin},
	urldate = {2023-03-13},
	date = {2019-11-01},
	eprinttype = {arxiv},
	eprint = {1802.06475 [math]},
}

@misc{canizo_harris-type_2021,
	title = {Harris-type results on geometric and subgeometric convergence to equilibrium for stochastic semigroups},
    year = {2021},
	url = {http://arxiv.org/abs/2110.09650},
	doi = { arXiv:2110.09650v1 },
	number = {{arXiv}:2110.09650},
	publisher = {{arXiv}},
	author = {Cañizo, José A. and Mischler, Stéphane},
	date = {2021-10-18},
	eprinttype = {arxiv},
	eprint = { arXiv:2110.09650v1 },
}

@article{EngelNagel2001,
author = {Engel, Klaus-Jochen and Nagel, Rainer},
year = {2001},
month = {06},
pages = {278-280},
title = {One-Parameter Semigroups for Linear Evolution Equations},
volume = {63},
journal = {Semigroup Forum},
doi = {10.1007/s002330010042}
}

@article{mischler_kacs_2014,
	title = {On {Kac}'s chaos and related problems},
	volume = {266},
	issn = {0022-1236},
	url = {https://www.sciencedirect.com/science/article/pii/S0022123614000925},
	doi = {10.1016/j.jfa.2014.02.030},
	language = {en},
	number = {10},
	urldate = {2023-07-11},
	journal = {Journal of Functional Analysis},
	author = {Hauray, Maxime and Mischler, Stéphane},
	month = may,
	year = {2014},
	pages = {6055--6157},
}

@article{goudon_junca_toscani_BE_2002,
	title = {Fourier-{Based} {Distances} and {Berry}-{Esseen} {Like} {Inequalities} for {Smooth} {Densities}},
	volume = {135},
	issn = {1436-5081},
	url = {https://doi.org/10.1007/s006050200010},
	doi = {10.1007/s006050200010},
	language = {en},
	number = {2},
	urldate = {2023-03-16},
	journal = {Monatshefte für Mathematik},
	author = {Goudon, Thierry and Junca, Stéphane and Toscani, Giuseppe},
	month = mar,
	year = {2002},
	pages = {115--136},
}

@article{carlen_entropy_2010,
	title = {Entropy and chaos in the {Kac} model},
	volume = {3},
	issn = {1937-5077},
	doi = {10.3934/krm.2010.3.85},
	language = {English},
	number = {1},
	journal = {Kinetic and Related Models},
	author = {Carlen, E.A. and Carvalho, M.C. and Le Roux, J. and Loss, M. and Villani, C.},
	year = {2010},
	pages = {85--122},
}

@article{esseen1942liapunov,
  title={On the Liapunov limit error in the theory of probability},
  author={Esseen, Carl-Gustav},
  journal={Ark. Mat. Astr. Fys.},
  volume={28},
  pages={1--19},
  year={1942}
}

@article{lions1995strengthened,
  title={A strengthened central limit theorem for smooth densities},
  author={Lions, Pierre Louis and Toscani, Giuseppe},
  journal={Journal of Functional Analysis},
  volume={129},
  number={1},
  pages={148--167},
  year={1995},
  publisher={Elsevier}
}

@article{ignat2007nonlocal,
  title={A nonlocal convection--diffusion equation},
  author={Ignat, Liviu I and Rossi, Julio D},
  journal={Journal of Functional Analysis},
  volume={251},
  number={2},
  pages={399--437},
  year={2007},
  publisher={Elsevier}
}

@article{auricchio2023,
title = {Trends to equilibrium for a nonlocal Fokker–Planck equation},
journal = {Applied Mathematics Letters},
volume = {145},
pages = {108746},
year = {2023},
issn = {0893-9659},
doi = {https://doi.org/10.1016/j.aml.2023.108746},
url = {https://www.sciencedirect.com/science/article/pii/S0893965923001787},
author = {Ferdinando Auricchio and Giuseppe Toscani and Mattia Zanella}
}

@article{cai2006stochastic,
  title={Stochastic protein expression in individual cells at the single molecule level},
  author={Cai, Long and Friedman, Nir and Xie, X Sunney},
  journal={Nature},
  volume={440},
  number={7082},
  pages={358--362},
  year={2006},
  publisher={Nature Publishing Group UK London}
}

@misc{ayi_structure-preserving_2022,
	title = {On a structure-preserving numerical method for fractional {Fokker}-{Planck} equations},
	url = {http://arxiv.org/abs/2107.13416},
	doi = {10.48550/arXiv.2107.13416},
	urldate = {2023-12-12},
	publisher = {arXiv},
	author = {Ayi, Nathalie and Herda, Maxime and Hivert, Hélène and Tristani, Isabelle},
	month = jul,
	year = {2022},
	note = {arXiv:2107.13416 [cs, math]},
}

@article{canizo2020hypocoercivity,
  title={Hypocoercivity of linear kinetic equations via harris's theorem},
  author={Ca{\~n}izo, Jos{\'e} A and Cao, Chuqi and Evans, Josephine and Yolda{\c{s}}, Havva},
  journal={Kinetic and Related Models},
  volume={13},
  number={1},
  pages={97--128},
  year={2020},
  publisher={American Institute of Mathematical Sciences}
}

@article{lods2008relaxation,
  title={Relaxation rate, diffusion approximation and Fick's law for inelastic scattering Boltzmann models},
  author={Lods, Bertrand and Mouhot, Cl{\'e}ment and Toscani, Giuseppe},
  journal={Kinetic and Related Models},
  volume={1},
  number={2},
  pages={223--248},
  year={2008}
}

@article{bisi2015entropy,
  title={Entropy dissipation estimates for the linear Boltzmann operator},
  author={Bisi, Marzia and Canizo, Jos{\'e} A. and Lods, Bertrand},
  journal={Journal of Functional Analysis},
  volume={269},
  number={4},
  pages={1028--1069},
  year={2015},
  publisher={Elsevier}
}

@article{dolbeauthypo,
title = {Hypocoercivity for kinetic equations with linear relaxation terms},
journal = {Comptes Rendus Mathematique},
volume = {347},
number = {9},
pages = {511-516},
year = {2009},
issn = {1631-073X},
doi = {https://doi.org/10.1016/j.crma.2009.02.025},
url = {https://www.sciencedirect.com/science/article/pii/S1631073X09000880},
author = {Jean Dolbeault and Clément Mouhot and Christian Schmeiser}
}

@article{villanihypo,
 author = {Villani, C.}, title = {Hypocoercivity}, journal = {Memoirs of the American Mathematical Society}, year = {2009}, volume = {202}, issue = {950}, pages = {0-0}, doi = {10.1090/s0065-9266-09-00567-5} }

@article{di_nezza_hitchhikers_2012,
	title = {Hitchhiker's guide to the fractional {Sobolev} spaces},
	volume = {136},
	issn = {0007-4497},
	url = {https://www.sciencedirect.com/science/article/pii/S0007449711001254},
	doi = {10.1016/j.bulsci.2011.12.004},
	number = {5},
	urldate = {2023-11-15},
	journal = {Bulletin des Sciences Mathématiques},
	author = {Di Nezza, Eleonora and Palatucci, Giampiero and Valdinoci, Enrico},
	month = jul,
	year = {2012},
	pages = {521--573},
}

@article{bucurvaldinoci2016
, author = {Bucur, C. and Valdinoci, E.}, title = {Nonlocal diffusion and applications}, journal = {Lecture Notes of the Unione Matematica Italiana}, year = {2016}, doi = {10.1007/978-3-319-28739-3} }

@article{caputo1967linear,
  title={Linear models of dissipation whose Q is almost frequency independent—II},
  author={Caputo, Michele},
  journal={Geophysical journal international},
  volume={13},
  number={5},
  pages={529--539},
  year={1967},
  publisher={Blackwell Publishing Ltd Oxford, UK}
}

@book{oberhettinger2013tabellen,
  title={Tabellen zur Fourier transformation},
  author={Oberhettinger, Fritz},
  volume={90},
  year={2013},
  publisher={Springer-Verlag}
}

@article{servadei2014spectrum,
  title={On the spectrum of two different fractional operators},
  author={Servadei, Raffaella and Valdinoci, Enrico},
  journal={Proceedings of the Royal Society of Edinburgh Section A: Mathematics},
  volume={144},
  number={4},
  pages={831--855},
  year={2014},
  publisher={Royal Society of Edinburgh Scotland Foundation}
}

@article{canizotassi2024,
title = {A uniform-in-time nonlocal approximation of the standard {F}okker-{P}lanck equation},
journal = {Discrete and Continuous Dynamical Systems},
volume = {46},
number = {0},
pages = {348-386},
year = {2026},
issn = {1078-0947},
doi = {10.3934/dcds.2025103},
url = {https://www.aimsciences.org/article/id/6867a988265e5f564963593a},
author = {José A. Cañizo and Niccolò Tassi}
}

@article{tristani2013fractional,
  title={Fractional fokker-planck equation},
  author={Tristani, Isabelle},
  journal={Commun. Math. Sci. 13, 5 , 1243--1260.},
  year={2015},
}

@inproceedings{mischler_uniform_2017,
	title = {Uniform semigroup spectral analysis of the discrete, fractional and classical {Fokker}-{Planck} equations},
	volume = {4},
	url = {https://jep.centre-mersenne.org/articles/10.5802/jep.46/},
	doi = {10.5802/jep.46},
	language = {en},
	urldate = {2023-07-03},
	booktitle = {Journal de l’École polytechnique — {Mathématiques}},
	author = {Mischler, Stéphane and Tristani, Isabelle},
	month = mar,
	year = {2017},
	note = {ISSN: 2270-518X},
	pages = {389--433},
}

@article{lafleche_fractional_2020,
author = {Laflèche, L.},
title = {Fractional fokker--planck equation with general confinement force},
journal = {SIAM Journal on Mathematical Analysis},
year = {2020},
volume = {52},
issue = {1},
pages = {164-196},
doi = {10.1137/18m1188331}
}

@article{gentil_levy-fokker-planck_2006,
	title = {The {Lévy}-{Fokker}-{Planck} equation: $\Phi$-entropies and convergence to equilibrium},
	volume = {59},
	shorttitle = {The {Lévy}-{Fokker}-{Planck} equation},
	doi = {10.3223/ASY-2008-0887},
	journal = {Asymptotic Analysis},
	author = {Gentil, Ivan and Imbert, Cyril},
	month = dec,
	year = {2006},
}

@article{nolan_stable_2020, 
author = {Nolan, J. P.},
title = {Univariate stable distributions},
journal = {Springer Series in Operations Research and Financial Engineering},
year = {2020},
doi = {10.1007/978-3-030-52915-4} }

@article{BourgainBrezisMironescu,
  title={Another look at Sobolev spaces},
  author={Bourgain, Jean and Brezis, Haim and Mironescu, Petru},
  year={2001},
journal={HAL-archive}
}

@article{cances2020large,
  title={Large time behavior of nonlinear finite volume schemes for convection-diffusion equations},
  author={Canc{\`e}s, Cl{\'e}ment and Chainais-Hillairet, Claire and Herda, Maxime and Krell, Stella},
  journal={SIAM Journal on Numerical Analysis},
  volume={58},
  number={5},
  pages={2544--2571},
  year={2020},
  publisher={SIAM}
}

@article{dujardin2020coercivity,
  title={Coercivity, hypocoercivity, exponential time decay and simulations for discrete Fokker--Planck equations},
  author={Dujardin, Guillaume and H{\'e}rau, Fr{\'e}d{\'e}ric and Lafitte, Pauline},
  journal={Numerische Mathematik},
  volume={144},
  number={3},
  pages={615--697},
  year={2020},
  publisher={Springer}
}

@article{bessemoulin2020hypocoercivity,
  title={Hypocoercivity and diffusion limit of a finite volume scheme for linear kinetic equations},
  author={Bessemoulin-Chatard, Marianne and Herda, Maxime and Rey, Thomas},
  journal={Mathematics of Computation},
  volume={89},
  number={323},
  pages={1093--1133},
  year={2020}
}

@article{bilerKarch_2003_generalizedFP,
  title={Generalized Fokker-Planck equations and convergence to their equilibria},
  author={Biler, Piotr and Karch, Grzegorz},
  journal={Banach Center Publications},
  volume={60},
  pages={307--318},
  year={2003},
  publisher={Polish Academy of Sciences}
}

@article{vazquez2017asymptoticbehaviourfractionalheat,
author = {Juan Luis Vázquez},
title = {Asymptotic behaviour for the fractional heat equation in the Euclidean space},
journal = {Complex Variables and Elliptic Equations},
volume = {63},
number = {7-8},
pages = {1216--1231},
year = {2018},
publisher = {Taylor \& Francis},
doi = {10.1080/17476933.2017.1393807},


URL = { 
    
        https://doi.org/10.1080/17476933.2017.1393807
    
    

},
eprint = { 
    
        https://doi.org/10.1080/17476933.2017.1393807
    
    

}

}

@Unpublished{hairer2010convergence,
  author = {Hairer, Martin},
  note   = {Unpublished lecture notes},
  title  = {Convergence of {Markov} processes},
  year   = {2010},
  url    = {https://www.hairer.org/notes/Convergence.pdf},
}

\end{document}